\documentclass[10pt,a4paper,reqno]{amsart} 
\usepackage{latexsym}
\usepackage{verbatim}
\usepackage{epsfig}
\usepackage{rotating}
\usepackage{amssymb}
\usepackage[T1]{fontenc}
\usepackage{afterpage}
\usepackage{color}
\usepackage{dsfont}
\usepackage{amssymb,amsfonts,amsmath,stmaryrd,bbm}
\usepackage[utf8]{inputenc}
\usepackage{pifont, tikz, subfigure}
\usetikzlibrary{plotmarks,shapes,arrows,positioning}
\usepackage{trees}
\usepackage{url}
\usepackage[plainpages=false,pdfpagelabels,colorlinks=true,citecolor=blue,hypertexnames=false]{hyperref}

\newcommand{\spacebreak}
{\begin{displaymath} \triangleleft \; \lhd \;
\diamond \; \rhd \; \triangleright
  \end{displaymath}}

\catcode`\@=11
\def\section{\@startsection{section}{1}%
 \z@{.7\linespacing\@plus\linespacing}{.5\linespacing}%
 {\normalfont\bfseries\scshape\centering}}

\def\subsection{\@startsection{subsection}{2}%
  \z@{.5\linespacing\@plus\linespacing}{.5\linespacing}%
  {\normalfont\bfseries\scshape}}

\def\subsubsection{\@startsection{subsubsection}{3}%
 \z@{.5\linespacing\@plus\linespacing}{-.5em}
 {\normalfont\bfseries}}
\catcode`\@=12

\newtheorem{theorem}{Theorem}
\newtheorem{lemma}[theorem]{Lemma}

\newtheorem{prop}[theorem]{Proposition}

\newtheorem{obs}[theorem]{Observation}
\newtheorem{definition}[theorem]{Definition}
\theoremstyle{definition}
\newtheorem{example}[theorem]{Example}

\let\set\mathbb

\renewcommand{\epsilon}{\varepsilon}
\renewcommand{\iota}{\phi}
\newcommand{\bone}{\bar 1}

\makeatletter
\def\testb#1{\testb@i#1,,\@nil}%
\def\testb@i#1,#2,#3\@nil{%
  \draw[->, thick] (O) --++(#1);
  \ifx\relax#2\relax\else\testb@i#2,#3\@nil\fi}
\makeatother   
\newcommand{\makediag}[1]{
    \coordinate (O) at (0,0); \coordinate (N) at (0,0.8);
    \coordinate (NE) at (0.8,0.8); \coordinate (E) at (0.8,0);
    \coordinate (SE) at (0.8,-0.8); \coordinate (S) at (0,-0.8);
    \coordinate (SW) at (-0.8,-0.8);\coordinate (W) at (-0.8,0);
    \coordinate (NW) at (-0.8,0.8); \coordinate (B1) at (1.2,1.2);
    \coordinate (B2) at (-1.2,-1.2);
    \testb{#1}
} 
\newcommand{\diagr}[1]{
  \begin{tikzpicture}[scale=0.8]\makediag{#1}\end{tikzpicture}
}

\def\clap#1{\hbox to0pt{\hss#1\hss}}

\def\Stepset#1#2#3#4#5#6#7#8#9{%
  \raisebox{-9pt}{%
  \setlength{\unitlength}{.9pt}%
  \begin{picture}(20,20)(-10,-10)
    \put(-5,-5){\clap{$\scriptstyle\ifx1#1\bullet\else\cdot\fi$}}
    \put(0,-5){\clap{$\scriptstyle\ifx1#2\bullet\else\cdot\fi$}}
    \put(5,-5){\clap{$\scriptstyle\ifx1#3\bullet\else\cdot\fi$}}
    \put(-5,0){\clap{$\scriptstyle\ifx1#4\bullet\else\cdot\fi$}}
    \put(0,0){\clap{$\scriptstyle\ifx1#5\bullet\else\cdot\fi$}}
    \put(5,0){\clap{$\scriptstyle\ifx1#6\bullet\else\cdot\fi$}}
    \put(-5,5){\clap{$\scriptstyle\ifx1#7\bullet\else\cdot\fi$}}
    \put(0,5){\clap{$\scriptstyle\ifx1#8\bullet\else\cdot\fi$}}
    \put(5,5){\clap{$\scriptstyle\ifx1#9\bullet\else\cdot\fi$}}
  \end{picture}}\StepsetB}

\def\StepsetB#1#2#3#4#5#6#7#8{%
  \raisebox{-8pt}{%
  \setlength{\unitlength}{.9pt}%
  \begin{picture}(20,20)(-10,-10)
    \put(-5,-5){\clap{$\scriptstyle\ifx1#1\bullet\else\cdot\fi$}}
    \put(0,-5){\clap{$\scriptstyle\ifx1#2\bullet\else\cdot\fi$}}
    \put(5,-5){\clap{$\scriptstyle\ifx1#3\bullet\else\cdot\fi$}}
    \put(-5,0){\clap{$\scriptstyle\ifx1#4\bullet\else\cdot\fi$}}
    \put(5,0){\clap{$\scriptstyle\ifx1#5\bullet\else\cdot\fi$}}
    \put(-5,5){\clap{$\scriptstyle\ifx1#6\bullet\else\cdot\fi$}}
    \put(0,5){\clap{$\scriptstyle\ifx1#7\bullet\else\cdot\fi$}}
    \put(5,5){\clap{$\scriptstyle\ifx1#8\bullet\else\cdot\fi$}}
  \end{picture}}\StepsetC}

\def\StepsetC#1#2#3#4#5#6#7#8#9{%
  \raisebox{-7pt}{%
  \setlength{\unitlength}{.9pt}%
  \begin{picture}(20,20)(-10,-10)
    \put(-5,-5){\clap{$\scriptstyle\ifx1#1\bullet\else\cdot\fi$}}
    \put(0,-5){\clap{$\scriptstyle\ifx1#2\bullet\else\cdot\fi$}}
    \put(5,-5){\clap{$\scriptstyle\ifx1#3\bullet\else\cdot\fi$}}
    \put(-5,0){\clap{$\scriptstyle\ifx1#4\bullet\else\cdot\fi$}}
    \put(0,0){\clap{$\scriptstyle\ifx1#5\bullet\else\cdot\fi$}}
    \put(5,0){\clap{$\scriptstyle\ifx1#6\bullet\else\cdot\fi$}}
    \put(-5,5){\clap{$\scriptstyle\ifx1#7\bullet\else\cdot\fi$}}
    \put(0,5){\clap{$\scriptstyle\ifx1#8\bullet\else\cdot\fi$}}
    \put(5,5){\clap{$\scriptstyle\ifx1#9\bullet\else\cdot\fi$}}
  \end{picture}}}

\def\sequenceThreeD#1#2#3#4{%
  #2\quad\hbox to.6\hsize{$#3,\dots$\hfill}\quad\textrm{(#4)}
}

\long\def\greybox#1{%
    \newbox\contentbox%
    \newbox\bkgdbox%
    \setbox\contentbox\hbox to \hsize{%
        \vtop{
            \kern\columnsep
            \hbox to \hsize{%
                \kern\columnsep%
                \advance\hsize by -2\columnsep%
                \setlength{\textwidth}{\hsize}%
                \vbox{
                    \parindent=0bp
                    #1
                }%
                \kern\columnsep%
            }%
            \kern\columnsep%
        }%
    }%
    \setbox\bkgdbox\vbox{
        \pdfliteral{0.9 0.9 0.9 rg}
        \hrule width  \wd\contentbox %
               height \ht\contentbox %
               depth  \dp\contentbox
        \pdfliteral{0 0 0 rg}
    }%
    \wd\bkgdbox=0bp%
    \vbox{\hbox to \hsize{\box\bkgdbox\box\contentbox}}%
}

\newcommand{\ns}{\mathbb{N}}

\newcommand{\zs}{\mathbb{Z}}

\newcommand{\qs}{\mathbb{Q}}

\newcommand{\fps}{formal power series}

\newcommand{\bx}{\bar x}
\newcommand{\by}{\bar y}
\newcommand{\bY}{\bar Y}
\newcommand{\bz}{\bar z}
\newcommand{\bX}{\bar X}

\newcommand{\bv}{\bar v}

\newcommand{\cC}{\mathcal C}
\newcommand{\cS}{\mathcal S}
\newcommand{\cU}{\mathcal U}
\newcommand{\cV}{\mathcal V}
\newcommand{\cW}{\mathcal W}
\newcommand{\cT}{\mathcal T}

\newcommand{\Sn}{\mathfrak S}
\DeclareMathOperator{\id}{id}

\newcommand{\p}{permutation}

\def\emm#1,{{\em #1}}
\newcommand{\si}{\sigma}

\def\hoeij#1{{#1}}

\begin{document}
\title[Walks confined to the positive octant]
{On 3-dimensional lattice walks\\
confined to the positive octant}

\author[A. Bostan]{Alin Bostan}

\author[M. Bousquet-M\'elou]{Mireille Bousquet-M\'elou}

\author[M. Kauers]{Manuel Kauers}

\author[S. Melczer]{Stephen Melczer}

\thanks{A.B. was supported by the Microsoft Research\,--\,Inria Joint Centre.}
\thanks{M.K. was supported by FWF Grants Y464-N18 and F50-04.}
\thanks{S.M. was supported by Inria, the Embassy of France in Canada, and an NSERC MSFSS}

\address{AB: INRIA Saclay,  1 rue Honoré d'Estienne d'Orves,
F-91120 Palaiseau, France} \email{Alin.Bostan@inria.fr}
 
\address{MBM: CNRS, LaBRI, Universit\'e de Bordeaux, 351 cours de la
  Lib\'eration,  F-33405 Talence Cedex, France} 
\email{bousquet@labri.fr}

\address{MK: RISC, Johannes Kepler Universität,
Altenbergerstraße 69, A-4040 Linz, Austria} 
\email{manuel@kauers.de}

\address{SM: Cheriton School of Computer Science, University of Waterloo, Waterloo ON
Canada  \&  U. Lyon, CNRS, ENS de Lyon, Inria, UCBL, Laboratoire LIP}
\email{smelczer@uwaterloo.ca}

\begin{abstract}

Many recent papers deal with the enumeration of 2-dimensional walks with
prescribed steps confined to the positive quadrant. The classification is now
complete for walks with steps in $\{0, \pm 1\}^2$: the generating function is
D-finite if and only if a certain group associated with the step set is
finite.

We explore in this paper the analogous problem for 3-dimensional walks
confined to the positive octant. The first difficulty is their number: we have
to examine no less than 11\,074\,225 step sets in $\{0, \pm 1\}^3$ (instead of
79 in the quadrant case). We focus on the 35\,548 that have at most six steps.

We apply to them a combined approach, first experimental and then rigorous. On
the experimental side, we try to guess differential equations. We also try to
determine if the associated group is finite. The largest finite groups that we
find have order 48 -- the larger ones have order at least 200 and we believe
them to be infinite. No differential equation has been detected in those
cases.

On the rigorous side, we apply three main techniques to prove D-finiteness.
The algebraic kernel method, applied earlier to quadrant walks, works in many
cases. Certain, more challenging, cases turn out to have a special \emm
Hadamard structure, which allows us to solve them via a reduction to problems
of smaller dimension. Finally, for two special cases, we had to resort to
computer algebra proofs. We prove with these techniques all the guessed
differential equations.

This leaves us with exactly 19 very intriguing step sets for which the group
is finite, but the nature of the generating function still unclear.
\end{abstract}

\maketitle

\section{Introduction}

The enumeration of lattice walks is a venerable topic in
combinatorics, which has numerous applications as well as connections
with other mathematical fields such as probability theory.
In recent years the enumeration of walks confined to cones has received a
lot of attention, and the present article also explores this topic.

Let us recall a few results for walks on $\zs^d$ that start at the origin
and consist of steps taken in~$\cS$, a finite subset of $\zs^d$. Clearly, there are $|\cS|^n$ 
walks of \emm length, $n$ (that is, using exactly $n$ steps), and the
associated (length) generating function,
$$
\sum_{w} t^{|w|}= \sum_{n\ge 0} |\cS|^nt^n =\frac 1{1-|\cS|t},
$$
is rational. Here, the sum runs over all walks~$w$, and $|w|$ denotes the
length of $w$.

If we now confine walks to a (rational) half-space, typically by enforcing the
first coordinate to be non-negative, then the resulting generating function
becomes algebraic. This question was notably studied in one dimension (walks
on the half-line
$\ns$)~\cite{gessel-factorization,Duchon98,bousquet-petkovsek-1,BaFl02}, but
most methods can be extended to arbitrary dimension. These methods are
effective and yield an explicit system of algebraic equations for the
generating function.

In the past few years, the ``next'' natural case, namely walks confined to the
intersection of two half-spaces, has been extensively studied in the form of
walks on $\zs^2$ confined to the positive quadrant
$\ns^2$~\cite{BoKa09,BoRaSa12,Bous05, BoMi10, FaIaMa99, FaRa10, KuRa11,
MiRe09}. We call such walks \emm quadrant walks., In the early days of this
study, it was sometimes believed, from the inspection of examples, that the
associated generating function was always \emm D-finite,, that is, satisfied a
linear differential equation with polynomial coefficients. This is now known
to be wrong, even for sets of \emm small, steps, that is, sets included in
$\{\bar 1 ,0,1\}^2$, with $\bar1=-1$. For example, the step set
$\cS=\{(1,1),(\bar 1 ,0),(0,\bar 1 )\}$ is associated with a D-finite (and
even algebraic) generating
function~\cite{gessel-proba,bousquet-versailles,Bous05}, but D-finiteness is
lost if the step $(\bone,\bone)$ is added~\cite{BoRaSa12}.

In fact, Bousquet-M\'elou and Mishna~\cite{BoMi10} observed that for quadrant
walks with small steps, the nature of the generating function seemed to be
correlated to the finiteness of a certain group of bi-rational transformations
of the $xy$-plane associated with $\cS$. Of the 79 non-equivalent step sets
(or: \emm models,\/) under consideration, they proved that exactly 23 sets
were associated with a finite group. Among them, they proved that 22 models
admitted D-finite generating functions. The 23rd model with a finite group was
proven D-finite (in fact, algebraic) by Bostan and Kauers~\cite{BoKa10}. These
D-finiteness results extend to the trivariate generating functions that keep
track of the length and the position of the endpoint. For 51 of the 56 models
with an infinite group, Kurkova and Raschel~\cite{KuRa11} proved that this
trivariate generating function is not D-finite, and Bostan \emph{et
al.}~\cite{BoRaSa12} proved that the (length) generating functions for walks
returning to the origin is not D-finite. The nature of the length generating
function\ for \emm all, quadrant walks is still unknown in those 51 cases.
The remaining 5 models were proven to have non-D-finite length generating
functions by Mishna and Rechnitzer~\cite{MiRe09} and Melczer and
Mishna~\cite{MeMi13}. 

\medskip
It is now natural to move one dimension higher, and to study  walks 
confined to the non-negative octant $\ns^3$. A similar group of
rational transformations can be defined: does the  satisfactory
dichotomy observed for quadrant walks (finite
group $\Leftrightarrow$ D-finite series) persist? Here, much less is
known. Beyond a few explicit examples, all we have so 
far is an empirical classification by Bostan and Kauers~\cite{BoKa09} of the
83\,682 (possibly trivial, possibly equivalent) models
$\mathcal{S}\subseteq\{\bar 1 ,0,1\}^3\setminus\{(0,0,0)\}$ with at most 
five steps. However, their paper does not discuss the associated group. 

The aim of the present article is to initiate a systematic study of octant
walks, their groups and their generating functions. We start from the 313\,912
models with at most six steps, and apply various methods in order to identify
provably and/or probably D-finite cases as well as probably non-D-finite
cases. All our strategies for proving D-finiteness are illustrated by
examples, and have been implemented. The software that accompanies this paper
can be found on \href{http://www.risc.jku.at/people/mkauers/bobokame}{Manuel
Kauers' website}.

We now give an overview of the paper. First, we discard models that are in
fact half-space models (and are thus algebraic), as well as models that are
naturally in bijection to others (Section~\ref{sec:IsoRed}). This leaves us
with (only) 35\,548 models with at most six steps. Some of them are equivalent
to a model of walks confined to the intersection of two half-spaces: we say
they are \emm two-dimensional,. The remaining ones are said to be \emm
three-dimensional,. The group of a model is defined at the end of
Section~\ref{sec:IsoRed}.

In Section~\ref{sec:summary} we describe our experimental attempts to decide,
for a given model, if the group is finite and if the generating function is
D-finite. By the end of the paper we will have proven all our D-finiteness
conjectures.

In Section~\ref{sec:kernel} we present the so-called \emm algebraic kernel
method,, a powerful tool that proves D-finiteness of many models with a finite
group. This is a natural extension of the method applied in~\cite{BoMi10} to
quadrant walks.

In Section~\ref{sec:hadamard} we show how the enumeration of walks of a \emm
Hadamard model, reduces to the enumeration of certain quadrant walks. Together
with the algebraic kernel method, this proves all our D-finiteness conjectures
for three-dimensional models (Section~\ref{sec:ThreeDim}). There remain 19
three-dimensional models with a finite group that may not be D-finite. Among
them, a striking example is the 3D analogue of Kreweras' quadrant
model~\cite{Bous05}.

In Section~\ref{sec:TwoDim} we discuss 2-dimensional octant models. To
simplify the discussion we only study their projection on the relevant
quadrant. This includes the ordinary quadrant models studied earlier, but also
models that have several copies of the same step. For all models with a finite
group, we prove the D-finiteness of the generating function. In two cases,
the only proofs we found are based on computer algebra. They are described in
Section~\ref{sec:computer}.

We conclude in Section~\ref{sec:final} with some comments and questions.

\section{Preliminaries}
\label{sec:IsoRed}

Let $\cS$ be a subset of $\{\bar 1 ,0,1\}^3\setminus\{(0,0,0)\}$, which we
think of as a set of \emm steps,, and often call \emm model,\footnote{Strictly
speaking, a model would be a step set plus a region to which walks are
confined, but the region will always be the non-negative octant in this paper,
unless specified otherwise.}. To shorten notation, we denote steps of $\zs^3$
by three-letter words: for instance, $\bone 10$ stands for the step
$(-1,1,0)$. We say that a step is $x$-positive (abbreviated as $x^+$) if its
first coordinate is 1. We define similarly $x^-$ steps, $y^+$ steps and so on.
Note that there are $2^{26}$ different models.

We define an \emm $\cS$-walk, to be any walk which starts from the origin
$(0,0,0)$ and takes its steps in $\cS$. The present focus is on $\cS$-walks
that remain in the positive octant $\ns^3$, with $\ns=\{0,1,2,\ldots\}$. We
are interested in the generating function that counts them by the length
(number of steps) and the coordinates of the endpoint: $$ O(x,y,z;t)=
\sum_{i,j,k,n \geq 0} o(i,j,k;n) x^iy^jz^kt^n, $$ where $o(i,j,k;n)$ is the
number of $n$-step walks in the octant that end at position $(i,j,k)$. The
dependence of our series on $t$ is often omitted, writing for instance
$O(x,y,z)$ instead of $O(x,y,z;t)$, and this series is called the \emm
complete, generating function of octant walks. We are particularly
interested in the \emm nature, of this series: it is \emm rational, if it can
be written as a ratio of polynomials, \emm algebraic, if there exists a
non-zero polynomial $P\in \qs[x,y,z,t,s]$ such that $P(x,y,z,t,O(x,y,z;t))=0$,
and \emm D-finite, (with respect to the variable $t$) if the vector space over
$\qs(x,y,z,t)$ spanned by the iterated derivatives $D_t^m O(x,y,z;t)$ has
finite dimension (here, $D_t$ denotes differentiation with respect to $t$).
The latter definition can be adapted to D-finiteness in several variables, for
instance $x$, $y$, $z$ and $t$: in this case we require D-finiteness with
respect to \emm each, variable separately~\cite{lipshitz-df}. Every rational
series is algebraic, and every algebraic series is D-finite.

For a ring $R$, we denote by $R[x]$ (resp.~$R[[x]]$) the ring of polynomials
(resp. formal power series) in $x$ with coefficients in $R$. If $R$ is a
field, then $R(x)$ stands for the field of rational functions in $x$. This
notation is generalised to several variables in the usual way. For instance,
$O(x,y,z;t)$ is a series of $\qs[x,y,z][[t]]$. Finally, if $F(u;t)$ is a power
series in $t$ whose coefficients are Laurent series in $u$, say, $$
F(u;t)=\sum_{n\ge0} t^n \left(\sum_{i\ge i(n)} u^i f(i;n)\right), $$ we denote
by $[u^{>0}]F(u;t)$ the \emm positive part of $F$ in $u$,: $$ [u^{>0}]F(u;t)=
\sum_{n\ge 0} t^n \left(\sum_{i>0 } u^i f(i;n)\right). $$ This series can be
obtained by taking a diagonal in a series involving one more variable:
\begin{equation}\label{diag} [u^{>0}]F(u;t) = \Delta_{s,t} \left( \frac s{1-s}
F(u/s, st)\right), \end{equation} where the (linear) diagonal operator $
\Delta_{s,t}$ is defined by $ \Delta_{s,t} (s^i t^j) = \mathbbm{1}_{i=j} t^j.$

\subsection{The dimension of a model}
\label{sec:dim}
Let $\cS$ be a model. A walk of length $n$  taking its steps in
$\mathcal{S}$ can be viewed as a word $w=w_1 w_2 \ldots w_n$ made up
of  letters  of $\cS$. For  $s\in
\cS$, let $a_s$ be the number of occurrences of $s$ in $w$ 
(also called multiplicity of $s$ in $w$).
Then $w$ ends in the positive octant if and only if 
the following three linear inequalities hold:
\begin{equation}\label{3ineq}
\sum_{s\in \cS} a_s s_x \ge 0, \quad \quad \sum_{s\in \cS} a_s s_y \ge 0,
\quad \quad \sum_{s\in \cS} a_s s_z \ge 0,  
\end{equation}
where $s=(s_x,s_y,s_z)$. Of course, the walk $w$ remains in the octant if the
multiplicities observed in each of its prefixes satisfy these inequalities.

\noindent
\begin{example}\label{ex:2D}
Take $\cS\ = \{0\bone\bar 1 , \bar 1 10, \bar 1 11,101\}$ (this is the third
model of Figure~\ref{fig:dimension}). If we write $a, b, c$ and $d$ for the
multiplicities of the four steps (taken in the above order), then the
inequalities~\eqref{3ineq} read
\[ d \geq b+c, \qquad\qquad\qquad b+c \geq a,\qquad\qquad\qquad c+d \geq a. \]
Note that if the first two inequalities hold, corresponding to a walk ending
in the intersection of two half-spaces, then the third inequality is
automatically satisfied. 
\end{example}

\begin{definition}\label{def:dim}
Let $d\in\{0,1,2,3\}$.  A model $\cS$ is said to have dimension at most
$d$ if there exist $d$ inequalities in~\eqref{3ineq} such that any
$|\cS|$-tuple $(a_s)_{s\in \cS}$ of non-negative integers satisfying
these $d$ inequalities satisfies in fact the three ones.
We define accordingly models of dimension (exactly) $d$.
\end{definition}
Examples are shown in Figure~\ref{fig:dimension}. 
\begin{figure}[bht]
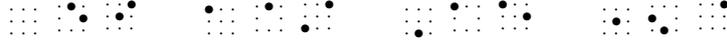

\begin{center}
\begin{tabular}{c@{\qquad}c@{\qquad}c@{\qquad}c}
\Stepset000 000 000   000 01 010 000 010 001 &
\Stepset000 000 100   000 00 010 100 000 001 &
\Stepset010 000 000   000 00 100 000 001 100 &
\Stepset00001000001010000000000001
\end{tabular}  
\end{center}
  \caption{Four-step models of respective dimension 0, 1, 2 and~3. For
  each model, the first diagram shows steps of the form $ij\bone$, the
  second shows steps $ij0$, and the third shows steps $ij1$.}
\label{fig:dimension}  
\end{figure}

It is clear that a model is 0-dimensional if and only if it is a
subset of $\{0,1\}^3\setminus\{000\}$. Such models have a rational
generating function:
$$ O(x,y,z;t)= \frac 1{1-t\sum_{s\in \cS}x^{s_x}y^{s_y}z^{s_z}}.$$
Let us also characterise models of dimension at most 1, or more precisely,
those in which the first inequality in~\eqref{3ineq} (also called the \emm
$x$-condition,) suffices to confine walks in the octant.

\begin{lemma}\label{lem:1D}
  Let  $\cS$ be a model. The  $y$- and $z$-conditions can be ignored 
when defining $\cS$-walks in the octant if and only if  the
following two conditions hold:
 \begin{enumerate}
 \item $\cS$ contains no  $y^-$ step or every   step $ijk\in \cS$ satisfies $j\ge i$,
 \item $\cS$ contains no  $z^-$ step
or every   step $ijk\in \cS$ satisfies $k\ge i$.
  \end{enumerate}
\end{lemma}
\begin{proof}
 Let us first recall the analogous result for quadrant
 walks~\cite[Section~2]{BoMi10}:  the $y$-condition can be ignored for
 a step set $\cS \subset \{\bone,0,1\}^2$ 
if and only if $\cS$ contains no $y^-$ step,
or  every step  $ij\in \cS$ satisfies $j\ge i$.

Let us now consider a model in $\zs^3$ such that the $y$- and $z$-conditions
can be ignored. In particular, as soon as an $\cS$-walk satisfies the
$x$-condition, it satisfies the $y$-condition. By projecting walks on the
$xy$-plane, and applying the above quadrant criterion, we see that either
there are no $y^-$ steps in $\cS$, or any step $ijk\in \cS$ satisfies $j\ge
i$. This is Condition (a). A symmetric argument proves~(b). Conversely, it is
easy to see that for any model satisfying~(a) and~(b), the $y$- and
$z$-conditions can be ignored. 
\end{proof}

For models of dimension at most 1, it suffices to enforce one of the three
conditions for the other two to hold. This means that we are effectively
counting walks confined to a half-space delimited, e.g., by the plane $x=0$.
As recalled in the introduction, such walks are known to have an algebraic
generating function, and the proof of algebraicity is constructive. For these
reasons, we do not discuss further 0- or 1-dimensional models.

For 2-dimensional models, we have not found an intrinsic characterisation that
would be the counterpart of Lemma~\ref{lem:1D}. In order to determine whether
a given model is (at most) 2-dimensional, we used integer linear
programming~\cite{Sch98}, as follows. For a model $\mathcal{S}$ of cardinality
$m$ define the three linear forms $I_1,I_2,I_3$ in the $m$ variables
$(a_s)_{s\in \cS}$ by~\eqref{3ineq}. One condition, say the third, is
redundant if the non-negativity of the corresponding linear form is implied by
the non-negativity of the two others, e.g., if for all $(a_s)_{s\in
\cS}\in\set N^m$ we have $(I_1\geq0)\land( I_2\geq 0)\Rightarrow I_3\geq0$. To
check algorithmically whether this is the case we minimise the objective
function $I_3$ under the constraints $I_1\geq0$ and $I_2\geq0$ (and $a_s\geq0$
for all~$s$). The third condition is redundant if and only if the minimum is
zero.

To check whether a given model has dimension at most two thus requires
checking whether 
\begin{multline*}
  (I_1\geq0)\land (I_2\geq0)\Rightarrow I_3\geq0\quad \hbox{ or } \quad(I_1\geq0)\land
(I_3\geq0)\Rightarrow I_2\geq0
\\
 \hbox{ or }\quad (I_2\geq0)\land
(I_3\geq0)\Rightarrow I_1\geq0.
\end{multline*}
This can be done by solving three small integer linear programming problems.
It is worth mentioning that in each 2-dimensional case that we discovered, the
2-dimensionality can be explained (up to a permutation of coordinates) by the
existence of two non-negative numbers $\alpha$ and $\beta$ such that for any
$ijk\in \cS$, $$ k\ge \alpha i +\beta j. $$ The pairs $(\alpha,\beta)$ that we
find are $(0,0), (1,0), (1,1), (1/2,1/2)$ and $ (1,2)$. In
Example~\ref{ex:2D}, this inequality holds for $\alpha=\beta=1$.

\subsection{Equivalent models}

Two models are said to be \emm equivalent, if they only differ by a
permutation of the step coordinates, or if they only differ by \emm unused
steps,, that is, steps that are never used in a walk confined to the octant.

\begin{lemma}\label{lem:unused}
    The model $\cS$ contains unused steps if and only if one of the following conditions holds:
  \begin{enumerate}
 \item [{$(A_x)$}] $\cS$ contains an $x^-$ step, but no $x^+$ step,
 \item [{$(B)$}] $\cS$ is non-empty, and each step of $\cS$ has a negative coordinate,
\item [{$(C_z)$}]$\cS$ contains $001$, any step $ijk$  of $\cS$
  satisfies $i+j\le 0$, and $\cS\not \subset \{001,00\bar 1 \}$,
 \item [{$(D)$}] there exists a permutation of the coordinates which, applied
   to $\cS$, gives a step set satisfying $A_x$ or $C_z$.
  \end{enumerate}
In Case~$A_x$, every $x^-$ step is unused. In Case B, all steps are unused. In
Case~$C_z$ all steps with a negative $x$- or $y$-coordinate are unused.
\end{lemma}

\begin{proof}

Let $s$ be an unused step. Let us prove that $s$ has at least one negative
coordinate, say along the $x$-direction, such that all $x^+$-steps (if any)
are also unused. We prove this by contradiction. Assume that for any negative
coordinate of $s$ there exists a used step which is positive in this
direction. Let $u_1, u_2, \ldots$ be a collection of such used steps (at most
three of them), one for each negative coordinate of $s$. Take an octant walk
ending with $u_1$, concatenate it with an octant walk ending with $u_2$, and
so on to obtain finally a walk $w$. Adding $s$ at the end of $w$ gives an
octant walk containing $s$. We have thus reached a contradiction, and proved
our statement.

So let us assume that $\cS$ contains an unused step with a negative
$x$-coordinate, and that all $x^+$ steps of $\cS$, if any, are unused. Then
all $x^-$ steps are unused. If there is no $x^+$ step, we are in Case $A_x$.
Now assume that there are $x^+$ steps, but that all of them are unused. Let us
then prove that $B$ or $C_z$ (or its variant $C_y$) holds.

\begin{itemize}
\item If each step of $\cS$ has a negative coordinate (Case $B$), then all
steps are indeed unused.
\item We now assume that $\cS$ contains a non-negative step, that is, a step
in $\{0,1\}^3$. The steps $111$, $110$, $101$ and $100$ cannot belong to
$\cS$, since they are $x^+$ and used. The step $011$ cannot belong to $\cS$
either, otherwise any $x^+$ step would be used.
\item We are thus left with sets in which the non-negative steps are $001$
and/or $010$. In fact, they cannot be both in $\cS$, otherwise any $x^+$ step
would be used. So assume that the only non-negative step of $\cS$ is $001$.
This forces every $x^+$ step to be $y^-$ (otherwise it would be used), and
conversely every $y^+$ step must be $x^-$ (otherwise any $x^+$ step would be
used). In other words, we are in Case $C_z$.
\end{itemize}
The rest of the lemma is obvious.
\end{proof}

\subsection{The number of non-equivalent models of dimension 2 or 3}
We now determine how many models are left when we discard  models
of dimension at most 1, models with unused steps, and when we moreover identify
models that only differ by a permutation of the coordinates. We count
these models by their cardinality.

\begin{prop}\label{prop:interesting}
   The generating function of  models
having dimension $2$ or $3$, no unused step, and counted up to
permutations of the coordinates, is  
\begin{multline*}
    I=73\,{u}^{3}+979\,{u}^{4}+6425\,{u}^{5}+28071\,{u}^{6}+91372\,{u}^{7}+
234716\,{u}^{8}+492168\,{u}^{9}\\+860382\,{u}^{10}+1271488\,{u}^{11}+
1603184\,{u}^{12}+1734396\,{u}^{13}+1614372\,{u}^{14}\\+1293402\,{u}^{15
}+890395\,{u}^{16}+524638\,{u}^{17}+263008\,{u}^{18}+111251\,{u}^{19}\\+
39256\,{u}^{20}+11390\,{u}^{21}+2676\,{u}^{22}+500\,{u}^{23}+73\,{u}^{
24}+9\,{u}^{25}+{u}^{26}.
\end{multline*}
\end{prop}

\noindent This is big (11\,074\,225 models), but more encouraging than
\[ (1+u)^{26}=
1+26\,u+325\,{u}^{2}+2600\,{u}^{3}+14950\,{u}^{4}+65780\,{u}^{5}+
230230\,{u}^{6}+O( {u}^{7}) . \] 
In particular, the number of non-equivalent interesting models of
cardinality at most 6 is 35\,548, which is about 9 times less than
$1+26+\cdots+230230=313\,912$. 

The proof of this proposition involves inclusion-exclusion, Burnside's
lemma, and the above characterisations of 0- and 1-dimensional
models and of models with unused steps. We first determine the
polynomial $J$ that counts,  up to 
permutations of the coordinates, sets having 
no unused step. We then subtract from $J$ the polynomial $K$ that
counts models of dimension at most 1.
The proof is rather tedious and given in Appendix~\ref{app:interesting}.

\subsection{The group of the model}
\label{sec:group}
Given a model $\cS$, we denote by $S$ the Laurent polynomial 
\begin{equation}\label{char-def}
S(x,y,z) =\displaystyle\sum_{ijk\in\mathcal{S}}x^iy^jz^k,
\end{equation}
and write  
\begin{align}
S(x,y,z) &= \bx 
A_{-}(y,z) + A_{0}(y,z) + x A_{+}(y,z)
\nonumber\\
             &= \by B_{-}(x,z) + B_{0}(x,z) + y B_{+}(x,z)\label{ABC-def}
\\
             &= \bz C_{-}(x,y) + C_{0}(x,y) + z C_{+}(x,y)\nonumber
\end{align}
where $\bx=1/x$, $\by=1/y$ and $\bz=1/z$.
We call $S$ the \emm characteristic polynomial, of~$\cS$.

Let us first assume that $\cS$ is 3-dimensional. Then it has
a positive step in each direction, and
 $A_+$, $B_+$ and $C_+$ are non-zero. 
The \emm group of, $\cS$ is the group $G$ of
birational transformations of the variables $[x,y,z]$ generated by the
following three involutions:
\[ \iota\left([x,y,z]\right) = \left[\bx\,
  \frac{A_{-}(y,z)}{A_+(y,z)},y,z\right] ,
\hspace{0.4in}
    \psi \left([x,y,z]\right) = \left[
      x,\by\, \frac{B_{-}(x,z)}{B_+(x,z)},z\right],
\]
\[
    \tau \left([x,y,z]\right) = \left[ x,y,\bz\, \frac{C_{-}(x,y)}{C_+(x,y)}\right]. \]
By construction, $G$ fixes the Laurent polynomial $S(x,y,z)$.

For a 2-dimensional model in which the $z$-condition can be ignored,
the relevant group is the group  generated by $\iota$ and $\psi$.

\section{Computer predictions and summary of the results}
\label{sec:summary}
We  focus in this paper on models with at most 6 steps.
Given such a model $\cS$, we compute the first 1000 terms of the series
$O(x_0,y_0,z_0;t)$, for all $(x_0,y_0,z_0) \in \{0,1\}^3$, using the following
step-by-step recurrence relation for the coefficients of $O(x,y,z;t)$: 
\begin{equation}\label{rec}
o(i,j,k;n)=\left\{
\begin{array}{ll}
  0 & \hbox{if } i<0 \hbox{ or } j<0 \hbox{ or } k<0,\\
\mathbbm{1}_{i=j=k=0} & \hbox{if } n=0,\\ 
\displaystyle \sum _{abc\in \cS} o(i-a,j-b,k-c;n-1) &\hbox{otherwise.}
\end{array}\right.
\end{equation}
From these numbers, we try to guess a differential equation (in $t$)
satisfied by $O(x_0,y_0,z_0;t)$
using the techniques described in~\cite{BoKa09}
and the references therein.

As described above in Section~\ref{sec:group}, we associate to $\cS$ a group
$G$ of rational transformations. Inspired by the quadrant case~\cite{BoMi10},
where the finiteness of the group is directly correlated to the D-finiteness
of the generating function, we also try to determine experimentally if $G$ is
finite. The procedure is easy: one simply writes out all words in $\iota$,
$\psi$, and $\tau$ of some length $N$, removes words which correspond to the
same transformation, and checks whether or not the remaining elements form a
group (if not then this process is repeated with a larger $N$).

\begin{figure}[htb]
  \centering
  \begin{tabular}{cccc}
     $ \diagr{NE,S,W}$  & $  \diagr{SW,N,E}$
  & $\diagr{N,S,E,W,NE,SW} $ &
  $ \diagr{E,W,SW,NE}$
\\
Kreweras & Reverse Kreweras& Double Kreweras& Gessel
  \end{tabular}
   \caption{The four quadrant models with orbit sum zero:
 all of them are algebraic.}
  \label{fig:alg}
\end{figure}
Before we summarise our results for octant walks, let us recall those obtained
for quadrant walks (Table~\ref{table2}). The most striking feature is of
course the equivalence between the finiteness of the group and the
D-finiteness of the generating function. The four trickiest D-finite models,
shown in Figure~\ref{fig:alg}, are those for which a certain \emm orbit sum,
(OS in the table) vanishes. They are in fact algebraic.

\begin{table}
\begin{center}
\def\alg{\leaf{\fbox{\textbf{algebraic}}}}
\def\dfinite{\leaf{\fbox{\textbf{D-finite}}}}
\def\notdfinite{\leaf{\fbox{\textbf{not D-finite}}}}
\edgeheight=5pt\nodeskip=1em\leavevmode
\tree{
  \begin{tabular}[c]{c}
    2D  quadrant models\\ $79=[7,23,27,16,5,1]$
  \end{tabular}
}{
        \tree{ \begin{tabular}[c]{c}$|G|{<}\infty$\\ $23=[5,6,4,5,2,1]$\end{tabular}}{
        \tree{ \begin{tabular}[c]{c}OS${\neq}0$\\ $19=[3,5,4,4,2,1]$\end{tabular}}{\dfinite}
        \tree{ \begin{tabular}[c]{c}OS${=}0$\\ $4=[2,1,0,1,0,0]$\end{tabular}}{\alg}
      }
      \tree{ \begin{tabular}[c]{c}$|G|{=}\infty$\\ $56=[2,17,23,11,3,0]$\end{tabular}}{\notdfinite}
    }
\end{center}
\caption{Results obtained for 2D quadrant
  walks. D-finiteness and algebraicity are meant
in all variables, $x$,
  $y$ and $t$ (see~\cite{BoMi10,BoKa10}). For the non-D-finite
  models, it is known that   $Q(0,0;t)$ or $Q(1,1;t)$ is not
  D-finite~\cite{BoRaSa12,MiRe09,MeMi13}. The numbers in brackets give
for each class the number of models of cardinality $3, 4, \ldots,
8$.}
\label{table2}
\end{table}

\begin{table}
\begin{center}
\def\dfinite{\leaf{\fbox{\textbf{D-finite}}}}
\def\notdfinite{\leaf{\fbox{\textbf{not D-finite?}}}}
\edgeheight=5pt\nodeskip=1em\leavevmode
\tree{\begin{tabular}{c}3D  octant models with $\leq6$ steps\\ $20804=[1,220,2852,17731]$\end{tabular}}
{
    \tree{
\begin{tabular}{c}
 $|G|{<}\infty$\\
$170=[0,26,47,97]$
\end{tabular}}{
      \tree{\begin{tabular}{c}OS${\neq}0$\\ $108=[0,11,31,66]$\end{tabular}}{\tree{kernel method}{\dfinite}}
      \tree{\begin{tabular}{c} OS${=}0$\\ $62=[0,15,16,31]$\end{tabular}}{
        \tree{\begin{tabular}{c}Hadamard\\ $43=[0,8,16,19]$\end{tabular}}{\dfinite}
        \tree{\begin{tabular}{c}not Hadamard\\ $19=[0,7,0,12]$\end{tabular}}{\notdfinite}
      }
    }
    \tree{\begin{tabular}{c}$|G|{=}\infty$?\\ $20634 =[1,194,2805,17634]$\end{tabular}}{\notdfinite}
  }
\end{center}
\caption{Results and conjectures for 3D models. D-finiteness is meant
  in all variables, $x, y,z$ and $t$. The numbers in brackets give for
  each class the number of models of cardinality 3, 4, 5 and 6.}
\label{table:3D}
\end{table}

Table~\ref{table:3D} summarises our conjectures and results for 3D octant
models. We find 170 groups of finite order (in fact, of order at most 48). The
remaining ones have order at least 200. All models for which the length
generating function $O(1,1,1;t)$ was conjectured D-finite have a finite
group, and we have in fact proved that their complete generating function
$O(x,y,z;t)$ is D-finite in its four variables. The table tells which of our
main two methods (the kernel method and the Hadamard decomposition) proves
D-finiteness. The main difference with the quadrant case is the following:
\emph{for 19 models with a finite group and zero orbit sum, we have not been
able to guess differential equations, and it may be that these models are not
D-finite.} Details are given in Section~\ref{sec:ThreeDim}. In particular, we
have computed much more than 1000 coefficients for these $19$ intriguing
models. Another difference between Tables~\ref{table2} and~\ref{table:3D} is
the disappearance of algebraic models. For some models $\cS$, we find certain
algebraic specialisations $O(x_0,y_0,z_0;t)$. But then the walks counted by
this series do not use all steps of $\cS$, and deleting the unused steps
leaves a model of lower dimension. We conjecture that, apart from these
degenerate cases, there is no transcendental series in 3D models. In
particular, we believe $O(x,y,z;t)$ to be transcendental.

For 2D octant models, we have only studied the projection of the walks on the
relevant quadrant. Since several 3D steps may project on the same 2D step,
this means studying a quadrant model with steps \emm in a multiset,. It seems
that the dichotomy established earlier for quadrant walks without multiple
steps still holds: all models for which we found a finite group have been
proved D-finite, and the others are conjectured to have an infinite group and
to be non-D-finite. See Section~\ref{sec:TwoDim} for details and a table
classifying these models. Four D-finite models, shown in
Figure~\ref{fig:quatre}, turn out to be especially interesting.

\section{The algebraic kernel method}
\label{sec:kernel}

In this section, we adapt to octant walks the material developed for quadrant
walks in Sections~3 and~4 of~\cite{BoMi10}. It will allow us to prove
D-finiteness for a large number of models.

\subsection{A functional equation}

Let $\mathcal{S}$ be a model with associated generating function $O(x,y,z;t)
\equiv O(x,y,z)$. Recall the definitions~\eqref{char-def} and~\eqref{ABC-def}
of the characteristic polynomial of $\cS$ and of the polynomials $A_+, A_0,
A_-$, etc.
We also let 
\[
D_{-}(z)=[\bx\by]
S(x,y,z) := \sum_{k \hbox{ \scriptsize{s.t. }} \bone\bone k \in \cS} z^k,
\] 
and define similarly
\[
 E_{-}(y) = [\bx\bz] S(x,y,z) \qquad \hbox{and } \quad F_{-}(x) = [\by\bz]S(x,y,z) .
 \] 
Finally, let $\epsilon$ be 1 if $\bar 1 \bar 1 \bar 1 $ belongs to
$\mathcal{S}$ and $0$ otherwise. The following functional equation translates
the fact that a walk of length $n$ must be a walk of length $n-1$ followed by
a step in $\mathcal{S}$, provided that this step does not take the walk out of
the octant:
\begin{align} \label{eq:fun3D}
O(x,y,z) = 1 &+ tS(x,y,z)O(x,y,z) \\
        &- {t}\bx A_{-}(y,z) O(0,y,z) - {t}{\by}B_{-}(x,z) O(x,0,z) - {t}{\bz}C_{-}(x,y) O(x,y,0) \notag \\
        &+ {t}{\bx\by}D_{-}(z) O(0,0,z)  + {t}{\bx\bz}E_{-}(y)O(0,y,0) + {t}{\by\bz}F_{-}(x) O(x,0,0)\notag \\
        &- {\epsilon t}{\bx\by\bz} O(0,0,0).\notag
\end{align}
 The terms on the last three lines ensure that the restriction to the
 positive octant is enforced, and are given by the inclusion-exclusion
 principle.  This is the series counterpart of the recurrence
 relation~\eqref{rec}. 
 
If the model is only 2-dimensional, so that, for instance, the
positivity condition in the third variable can be ignored, the
following simpler equation holds:
\begin{align} \label{eq:fun3D-2D}
O(x,y,z) = 1 &+ tS(x,y,z)O(x,y,z) \\
        &- {t}{\bx}A_{-}(y,z) O(0,y,z) - {t}{\by}B_{-}(x,z)
        O(x,0,z) 
 \notag \\
        &
+{t}{\bx\by}D_{-}(z) O(0,0,z) \notag .
\end{align}
It is then of the same nature as the equations studied for quadrant
models, with the variable $z$ playing no particular
role~\cite{BoMi10}. 

\subsection{Orbit sums for finite groups} 
Let us multiply~\eqref{eq:fun3D} by $xyz$ and group the terms
involving $O(x,y,z)$. This gives
\begin{align} \label{eq:ker3D}
xyzK(x,y,z)O(x,y,z) &= xyz - tyzA_{-}(y,z) O(0,y,z)\\
& - txzB_{-}(x,z) O(x,0,z) - txyC_{-}(x,y) O(x,y,0)\notag \\
        &+ t zD_{-}(z) O(0,0,z)  + tyE_{-}(y)O(0,y,0) + txF_{-}(x) O(x,0,0)\notag \\
        &- t\epsilon O(0,0,0).\notag
\end{align}
where $K(x,y,z) := 1-tS(x,y,z)$ is the \textit{kernel} of the model. If the
model is 2-dimensional (with the $z$-condition redundant), the terms in $C_-$,
$E_{-}$, $F_-$ and $\epsilon$ are not there, see~\eqref{eq:fun3D-2D}. What is
important in the above equation is that all unknown summands on the right-hand
side involve at most two of the three variables $x$, $y$ and $z$. For
instance, the second term, $tyzA_-(y,z)O(0,y,z)$, does not involve $x$.

Assume that the group $G$ is finite. Let us write the $|G|$ equations obtained
from~\eqref{eq:ker3D} by replacing the 3-tuple $(x,y,z)$ by any element of its
orbit. Then, we form the alternating sum of these $|G|$ equations by weighting
each of them by the sign of the corresponding group element (that is, minus
one to the length of a minimal word in $\iota, \psi$ and $\tau$ that expresses
it). Since each unknown term on the right-hand side involves at most two of
the three variables $x,y,z$, and each of the involutions $\iota$, $\psi$ and
$\tau$ fixes one coordinate, all unknown series cancel on the right-hand side,
leaving:
\begin{equation} \label{eq:sum3D}
\sum_{g \in G} \text{sign}(g)\  g\!\left(xyzO(x,y,z)\right) = \frac{1}{K(x,y,z)}  \sum_{g \in G} \text{sign}(g)g(xyz),
\end{equation}
where, for any series $T(x,y,z)$ and any element $g$ of $G$,
the term $g(T(x,y,z))$ must be understood as $T(g(x,y,z))$. 
The right-hand side is now an explicit rational function, called the
\emm orbit sum,.
In the left-hand side, we have a \fps\ in $t$ with coefficients in
$\qs(x,y,z)$. We call~\eqref{eq:sum3D} the \emm orbit equation,.

Assume that all the coordinates of all elements in the orbit of $[x,y,z]$ are
\emm Laurent polynomials, in $x$, $y$ and $z$. This happens for instance in
the example of Section~\ref{sec:example-kernel-Laurent} below (the orbit is
shown in Figure~\ref{fig:orbit}). Assume moreover that for any element
$[x',y',z']$ in the orbit, other than $[x,y,z]$, there exists a variable, say
$x$, such that $x'$, $y'$ and $z'$ are in fact polynomials in $\bx$. This is
again true in Figure~\ref{fig:orbit}. Then the series $x'y'z'O(x',y',z')$
occurring in the left-hand side of~\eqref{eq:sum3D} consist of monomials $x^i
y^j z^k t^n$ where at least one of $i, j$ and $k$ is non-positive. In this
case, extracting from~\eqref{eq:sum3D} the monomials with positive exponents
in $x, y$ and $z$ gives
\begin{equation}\label{extracted}
xyzO(x,y,z;t) = [x^{>0}][y^{>0}][z^{>0}]\frac{1}{K(x,y,z;t)} \sum_{g
  \in G} \text{sign}(g)g(xyz).
\end{equation}
By~\eqref{diag}, the positive part of a series can be expressed as a
diagonal. Thus $O(x,y,z;t)$ is a diagonal of a rational
series, and is D-finite in its four variables~\cite{Li88,lipshitz-df}.  

More generally, the above identity holds if every orbit element $[x',y',z']$
other than $[x,y,z]$ satisfies:
\begin{multline}
\label{extraction}
\{x',y',z'\} \subset \qs(x,y)[\bz]
\quad \hbox{or } \quad 
\{x',y',z'\} \subset \qs(x)[\by,z]
\quad \\ \hbox{or } \quad 
\{x',y',z'\} \subset \qs
 [\bx,y,z]. \quad
\end{multline}
Indeed, in this case, the coefficient of $t^n$ in $x'y'z'O(x',y',z';t)$, once
expanded as a (Laurent) series in $z$, then $y$, then $x$, consists of
monomials with at least one non-positive exponent. Extracting the monomials
where all exponents are positive gives~\eqref{extracted}. An example is
detailed in Section~\ref{sec:example-kernel-rational}.

Condition~\eqref{extraction} does not cover all cases where the extraction of
monomials with positive exponents in $x$, $y$ and $z$ leads
to~\eqref{extracted}. Two interesting two-dimensional examples are discussed
in Section~\ref{sec:kernel-interesting}.

This procedure is what we call the \emm algebraic kernel method,.

\subsection{First illustration of the kernel method}
\label{sec:example-kernel-Laurent}

Consider the step set $\cS=\{\bar 1 \bar 1 \bar 1 ,\bar 1 \bar 1 1,\bar 1
10,100\}$. Denoting by $a, b, c$ and $d$ the number of steps of each type in
an $\cS$-walk, the positivity conditions read
$$ d\ge a+b+c, \quad c\ge a+b, \quad b\ge a, $$
and it is clear that none is redundant. So the model is 3-dimensional.
The functional equation~\eqref{eq:fun3D} reads in this case
\begin{multline*}
   K(x,y,z)  O(x,y,z)= 1  
-t\bx(y+\by z+\by \bz) O( 0,y,z)
-t\bx\by (z+\bz) O(x, 0,z)
-t\bx\by\bz O(x,y, 0)\\
+t\bx\by (z+\bz) O( 0, 0,z)
+ t\bx\by\bz O( 0, y,  0)
+t\bx\by\bz O(x, 0, 0)
-t\bx\by\bz O( 0, 0, 0),
\end{multline*}
where the kernel is
$$ K(x,y,z)=1-t( \bx\by\bz +\bx\by z+\bx y+x). $$
The images of $[x,y,z]$ by the involutions $\iota$, $\psi$ and $\tau$
are respectively
$$
\left[\bx(y+\by z+\by\bz),y,z\right],
\quad \left[x,\by(z+\bz),z\right],
\quad \left[x,y,\bz\right].
$$
By composing these involutions, one observes that they generate a group of
order~8, shown in Figure~\ref{fig:orbit}. In fact, $\iota$, $\psi$ and $\tau$
commute.
Note that all coordinates in the orbit are Laurent polynomials in $x$, $y$
and $z$. The orbit equation~\eqref{eq:sum3D} reads
\begin{multline*}
  xyzO(x,y,z) - \bx yz(y+\by z+\by \bz) O(\bx(y+\by z+\by \bz),y,z )
- x\by z(z+\bz) O(x, \by(z+\bz),z)\\
-xy\bz O(x,y,\bz)
+\bx \by z (y+\by z+\by \bz)(z+\bz) O(\bx(y+\by z+\by
\bz),\by(z+\bz),z)
\\
+\bx y\bz(y+\by z+\by\bz) O(\bx(y+\by z+\by \bz),y,\bz)
+x\by\bz(z+\bz)O(x,\by(z+\bz),\bz)
\\
-\bx\by\bz(y+\by z+\by\bz)(z+\bz) O(\bx(y+\by z+\by \bz),\by(z+\bz),\bz)
\\=
\frac{(x-\bx y-\bx\by z-\bx\by\bz)(y-\by z-\by\bz)(z-\bz)}{xyzK(x,y,z)}.
\end{multline*}
Now let us examine the eight unknown series occurring in
the left-hand side. The series $xyzO(x,y,z)$ is \emm positive, in $x$,
$y$ and $z$, meaning that all its monomials  involve a
positive power of each variable. In contrast, each of the other seven
series is non-positive (and in fact, negative) in at least one
variable, because Condition~\eqref{extraction} holds.
Hence, extracting the positive part in the above identity gives
$$
xyzO(x,y,z)= [x^{>0}] [y^{>0}] [z^{>0}]\frac{(x-\bx y-\bx\by
  z-\bx\by\bz)(y-\by z-\by\bz)(z-\bz)}{xyz(1-t( \bx\by\bz +\bx\by z+\bx y+x))}.
$$
We thus conclude that $O(x,y,z;t)$ is D-finite in all variables. Moreover, the
coefficient extraction can be performed explicitly, and gives rise to nice
coefficients. Details of the extraction are left to the reader.

\begin{prop}
  For the model $\cS=\{\bar 1 \bar 1 \bar 1 ,\bar 1 \bar 1 1,\bar 1 10,100\}$,
 the number of walks of length $n$ ending
at $(i,j,k)$ is non-zero if
and only if $n$ can be written $8m+i+2j+4k$, 
in which case
$$
o(i,j,k;n)=\frac{(i+1)(j+1)(k+1)n!}{(4m+i+j+2k+1)!(2m+j+k+1)!(m+k+1)!m!}.
$$
\end{prop}

\noindent{\bf Note.} Throughout the paper, $m!$ is defined to be infinite if $m<0$. 

\begin{figure}[ht]
\begin{center}
\hskip -10mm\scalebox{0.9}{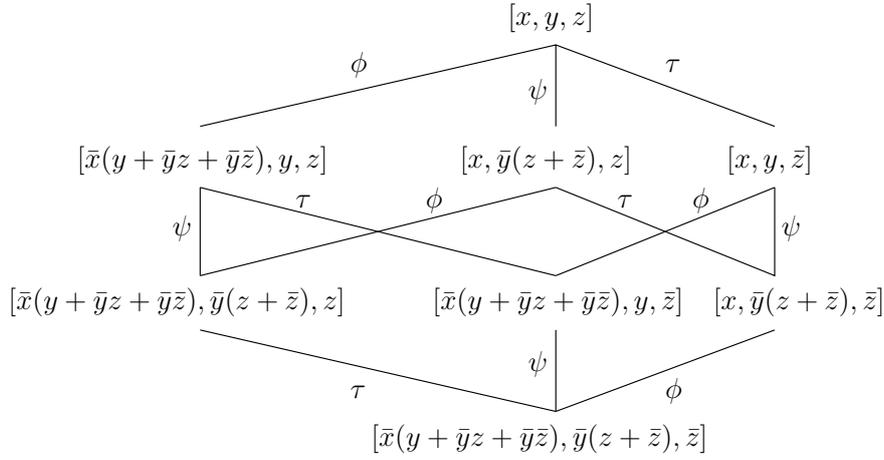}
\caption{The orbit of $[x,y,z]$ under the group generated by $\iota$,
  $\psi$ and $\tau$, when $\cS=\{\bar 1 \bar 1 \bar 1 ,\bar 1 \bar 1 1,\bar 1 10,100\}$.} 
\label{fig:orbit}
\end{center}
\end{figure}

\subsection{Second illustration of the kernel method}
\label{sec:example-kernel-rational}
Let us now take
$\cS=\{\bone 0\bone, \bone 11, 0 \bone 1, 10\bone, 111\}$. This model
is also 3-dimensional.  The functional equation~\eqref{eq:fun3D} reads
\begin{multline*}
K(x,y,z)O(x,y,z)=1-t\bx(\bz+yz)O(0,y,z)
-t\by zO(x,0,z)\\
-t\bz(x+\bx) O(x,y,0)
+t\bx\bz O(0,y,0),
\end{multline*}
where
$$
K(x,y,z)= 1-t\left( \bx\bz+\bx yz+\by z+x\bz+xyz\right).
$$
The images of $[x,y,z]$ by the involutions $\phi$, $\psi$ and $\tau$
are respectively
$$
\left[ \bx,y ,z \right], \quad \left[x ,\frac\by{x+\bx} , z\right],
\quad \left[x ,y , \bz\frac{x+\bx}{\by+y(x+\bx)}\right] .
$$
As in the previous example, these three involutions commute, and thus
generate a group of order 8. The coefficients of the series occurring
in the orbit equation are rational functions in $x, y$ and $z$, but no
longer Laurent polynomials:
\begin{multline*}
  xyzO\left(x,y,z\right)
-\bx y z O\left(\bx,y ,z \right)
-\frac{x\by z}{x+\bx} O\left(x  ,\frac\by{x+\bx} , z\right)
\\
-\frac{x y  \bz\left(x+\bx\right)}{\by+y\left(x+\bx\right)} O\left( x ,y , \frac{\bz(x+\bx)}{\by+y\left(x+\bx\right)}\right)
+ \frac{\bx \by z}{x+\bx}  O\left(\bx  ,\frac\by{x+\bx} , z\right)
\\
+\frac{\bx y  \bz(x+\bx)}{\by+y\left(x+\bx\right)}O\left(\bx ,y , \frac{\bz(x+\bx)}{\by+y\left(x+\bx\right)}\right)
+\frac{x \by\bz}{\by+y\left(x+\bx\right)}O\left( x
  ,\frac\by{x+\bx},\frac{\bz(x+\bx)}{\by+y\left(x+\bx\right)}\right)
\\
-\frac{\bx \by\bz}{\by+y\left(x+\bx\right)} O\left(\bx
  ,\frac\by{x+\bx},\frac{\bz(x+\bx)}{\by+y\left(x+\bx\right)}\right)
\\=
\frac{x-\bx}{x+\bx} \cdot \frac {\by-y\left(x+\bx\right)}{\by+y\left(x+\bx\right)}\cdot 
\frac {\bx\bz-\bx yz-\by z+x\bz-xyz}{K(x,y,z)}.
\end{multline*}
However, we observe that for any series $O(x',y',z')$ occurring in
this equation, other than $O(x,y,z)$, the variables $x', y', z'$
satisfy Condition~\eqref{extraction}.
By extracting the positive part first in $z$, then in $y$, and finally
in $x$, we thus obtain:
\begin{multline*}
  xyzO\left(x,y,z\right)
=
[x^{>0}][y^{>0}][z^{>0}]\frac{x-\bx}{x+\bx} \cdot \frac {\by-y\left(x+\bx\right)}{\by+y\left(x+\bx\right)}\cdot 
\frac {\bx\bz-\bx yz-\by z+x\bz-xyz}{K(x,y,z)},
\end{multline*}
so that $O(x,y,z;t)$ is again D-finite. Moreover, the coefficient
extraction can be performed explicitly and gives rise to nice
numbers. Details are left to the reader.

\begin{prop}
  For the model $\cS=\{\bone 0\bone, \bone 11, 0 \bone 1, 10\bone, 111\}$,
the number of walks of length $n$ ending at $(i,j,k)$ is non-zero if
and only if $n$ can be written $8m+4i+2j+3k$, in which case
\begin{multline*}
  o(i,j,k;n)= \frac { \left( i+1 \right)  \left( j+1 \right)  \left( k+1 \right)}{  \left( 4m+2i+j+2k+1 \right)}\frac{ \left( 6m+3i+2j+2k \right) !}{\left( 3m+2i+j+k+1 \right) ! \left( 3m+i+j+k \right) !} \\
\frac{ n!}{ \left( 2m+i+k \right) ! \left( 2m+i+j+k+1 \right) !
  \left( 4m+2i+j+k \right) ! }.
\end{multline*}
\end{prop}

\subsection{When the orbit equation does not suffice}

In the study of quadrant walks, there are four models with a finite group for
which the complete generating function, denoted by $Q(x,y)$, cannot be
extracted from the orbit equation
$$
\sum_{g \in G} \text{sign}(g)\  g\!\left(xyQ(x,y)\right) = \frac{1}{K(x,y)}  \sum_{g \in G} \text{sign}(g)g(xy).
$$
They are shown in Figure~\ref{fig:alg}. In each case, the orbit sum happens to
be $0$. This implies that $Q(x,y)=1$ solves the orbit equation, which does not
determine $Q(x,y;t)$ uniquely. Three of these four models can be solved using
a half-orbit sum, followed by a more delicate extraction
procedure~\cite{BoMi10}. The fourth one was first solved using intensive
computer algebra~\cite{KaKoZe08,BoKa10}, and more recently using complex
analysis~\cite{BoKuRa13}.

For models in $\zs^3$, it also happens that we find a finite group but a
non-conclusive orbit equation. This happens for 62 three-dimensional models,
and then the orbit sum is zero\footnote{For 2D octant models, we also find one
non-conclusive orbit equation with a non-zero orbit sum, see
Section~\ref{sec:p1} for details.}. This is the case for instance with
$S(x,y,z)=\bx+ xyz+x\by +x\bz$, or for the 3-dimensional analogue of Kreweras'
walks, $S'(x,y,z)=\bx+\by+\bz+xyz$. For 43 of these 62 models, including $S$,
we were able to guess differential equations, and we have then proved the
D-finiteness of $O(x,y,z;t)$ via a combinatorial construction described in
Section~\ref{sec:hadamard} below. For the 19 others models, including $S'$, we
have not been able do guess any differential equation nor to prove
D-finiteness. We refer to Section~\ref{sec:nonHzero} for details on these 19
models.

\section{Hadamard walks}\label{sec:hadamard}

In this section, we describe how the study of some 3-dimensional models ---
called \emm Hadamard models, --- can be reduced to the study of a pair of
models, one in $\zs$ and the other in $\zs^2$. Let us begin with an example.

\subsection{Example}
\label{sec:H-example}
Let $\cS$ be the  step set with characteristic polynomial
$$ S(x,y,z)=x+ (1+x+\bx)(yz+\by+\bz). $$
The associated group is finite of order 12. The orbit of $[x,y,z]$ contains
$[x,z,y]$, and thus the orbit sum is zero. We will prove the D-finiteness of
the series $O(x,y,z;t)$ using a combinatorial argument.

Write $U=x$, $V=1+x+\bx$, $T=yz+\by+\bz$, so that $S=U+VT$. Let $\cU, \cV$ and
$\cT$ denote the corresponding step sets, respectively in $\{\bar 1 , 0, 1\}$
(for $\cU$ and $\cV$) and $\{\bar 1 , 0, 1\}^2$ (for $\cT$). We thus have
$$ \cS=\left(\cU\times \{0\}^2\right) \cup \left(\cV\times \cT\right), $$
and the union is disjoint.

Consider now an octant walk $w$ with steps in $\cS$, and project it on the
$x$-axis. This gives a walk on $\ns$ with steps $\bar 1 $, $0$ and $1$. Colour
white the steps that are projections of a step of $\cU\times \{0\}^2$ (that
is, of a step $100$), and colour black the projections of the steps of
$\cV\times \cT$. This gives a coloured walk $w_1$. Now return to $w$ and
delete the steps of $\cU\times \{0\}^2$ (that is, the steps $100$). The
resulting walk lives in $\zs\times \ns^2$. Project it on the $yz$-plane to
obtain a second walk $w_2$, which has steps in $\cT$ and is a quarter plane
walk.

The walk $w$ can be recovered from $w_1$ and $w_2$ as follows. Start from
$w_1$, and leave the white steps unchanged. Replace the $j$th black step of
$w_1$, with value $a\in \{\bar 1 ,0,1\}$, by $abc$, where $bc\in \cT$ is the
$j$th step of $w_2$.

Conversely, let $w_1$ be a walk in $\ns$ with steps in $\cU \cup \cV$ having
black and white steps, such that all steps of $\cU\setminus \cV$ are white and
all steps of $\cV\setminus \cU$ are black. In other words, the only steps for
which we can choose the colour are those of $\cU\cap \cV$. Let $w_2$ be a walk
in $\ns^2$ with steps in $\cT$, whose length coincides with the number of
black steps in $w_1$. Then the walk $w$ constructed from $w_1$ and $w_2$ as
described above is an octant walk with steps in $\cS$.

Let $C_1(x,v;t)$ be the generating function of coloured walks like $w_1$,
counted by the length ($t$), the number of black steps ($v$) and the
coordinate of the endpoint ($x$). Let $C_2(y,z;v)$ be the generating
function of walks like $w_2$, counted by the length ($v$) and the
coordinates of the endpoint ($y,z$). The above construction shows that
$$ O(x,y,z;t)= \left. C_1(x,v;t) \odot_v C_2(y,z;v)\right|_{v=1}, $$
where $\odot_v$ denotes the Hadamard product with respect to $v$:
$$ \sum_i a_iv^i \odot_v \sum_j b_j v^j= \sum_i a_ib_i v^i, $$
and the resulting series is specialised to $v=1$.

In our example, $C_1(x,v;t)$ counts (coloured) walks on $\ns$ and is easily
seen to be algebraic, while $C_2(y,z;v)$ is the generating function of
Kreweras walks, which is also algebraic~\cite{Bous05}. Since the Hadamard
product preserves D-finiteness~\cite{Li88}, we conclude that $O(x,y,z;t)$ is
D-finite.

\subsection{Definition and enumeration of Hadamard walks}
\label{sec:hadam1}

We can now generalise the above discussion to count octant walks for \emm
Hadamard models,. Since this discussion works in all dimensions, we actually
consider walks in $\ns^D$, starting from the origin and taking their steps in
a set $\cS\subset\{\bar 1 ,0,1\}^D\setminus\{(0, \ldots, 0)\}$. We denote by
$S(x_1, \ldots, x_D)$ the characteristic polynomial of $\cS$:
$$
S(x_1, \ldots, x_D)= \sum_{(i_1, \ldots, i_D)\in \cS} x_1^{i_1} \cdots
x_D^{i_D}.
$$
We also denote by $0^i$ the $i$-tuple $(0, \ldots, 0)$. Assume there exist
positive integers $d$ and $\delta$ with $d+\delta=D$, and three sets $\cU
\subset \{\bar 1 ,0,1\}^d\setminus \{0 ^d \}$, $\cV\subset \{\bar 1 ,0,1\}^d$
and $\cT \subset \{\bar 1 ,0,1\}^\delta\setminus \{0^\delta\}$, such that
\begin{equation}\label{hadamard-def}
\cS= \left( \cU \times \{0^\delta\}\right) \cup \left(\cV\times \cT\right).
\end{equation}
Note that the union is necessarily disjoint. The characteristic polynomial of
$\cS$ reads
$$
S(x_1, \ldots, x_D)= U(x_1, \ldots, x_d) + V(x_1, \ldots,
x_d)T(x_{d+1}, \ldots, x_D).
$$
We say that $\cS$ is  $(d,\delta)$\emm-Hadamard,. 

Let $ \cC_1$ be the set of walks with steps in $\cU \cup \cV$ confined to
$\ns^d$, in which the steps are coloured black and white, with the condition
that all steps of $\cU\setminus \cV$ are white and all steps of $\cV\setminus
\cU$ are black. We call these walks \emm coloured, $(\cU, \cV)$\emm-walks,.
Let $C_1(x_1, \ldots, x_d,v;t)$ be the associated generating function, where
$t$ keeps track of the length, $x_1, \ldots, x_d$ of the coordinates of the
endpoint, and $v$ of the number of black steps. Equivalently,
\begin{equation}\label{C1}
C_1(x_1, \ldots, x_d,v;t)=\sum_w x_1^{i_1(w)} \cdots
x_d^{i_d(w)}(1+v)^{|w|_{\cU \cap \cV}}v^{|w|_{\cV\setminus \cU}} t^{|w|},
\end{equation}
where the sum runs over all \emm uncoloured, walks $w$ in $\ns^d$ with steps
in $\cU\cup\cV$, the values $i_1(w), \ldots,i_d(w)$ are the coordinates of the
endpoint and $|w|_\cW$ stands for the number of steps of $w$ belonging to
$\cW$, for any step set $\cW\subset \cU\cup\cV$. Let $C_2(x_{d+1}, \ldots,
x_D;v)$ be the generating function of $\cT$-walks confined to $\ns^\delta$,
counted by the length ($v$) and the coordinates of the endpoint ($x_{d+1},
\ldots, x_D$).
\begin{prop}\label{prop:hadamard}
Assume $\cS$ is $(d,\delta)$-Hadamard, given
by~\eqref{hadamard-def}. With the above notation, the series $O(x_1, \ldots,
x_D;t)$ that counts $\cS$-walks confined to $\ns^D$ is 
$$
O(x_1, \ldots,x_D;t)= \left. C_1(x_1, \ldots, x_d,v;t) \odot_v
  C_2(x_{d+1}, \ldots, x_D;v)\right|_{v=1},
$$
where $\odot_v$ denotes the Hadamard product with respect to $v$.

In particular, $O$ is D-finite (in all its variables) if $ C_1$ and $C_2$ are
D-finite (in all their variables).
\end{prop}
The proof is a direct extension of the argument given in
Section~\ref{sec:H-example}, and is left to the reader. We will
also use the following simple observation.
\begin{obs}\label{mixt}
 The generating function of $\cS$-walks   confined to $\ns^d\times
\zs^\delta$, counted by the length and the coordinates of the
endpoint, is
$$ C_1\left(x_1, \ldots, x_d, T(x_{d+1}, \ldots, x_D);t\right). $$
\end{obs}
We now  specialise the Hadamard decomposition to octant walks.

\subsection{The case of $\boldsymbol{(1,2)}$-Hadamard walks}
Assume that the model $\cS$ is $(1,2)$-Hadamard:
$$ S(x,y,z)= U(x)+ V(x)T(y,z). $$
This was the case with the model of Section~\ref{sec:H-example}. Since the
sets $\cU$ and $\cV$ are (at most) 1-dimensional, the generating function $
C_1(x,v;t)$ counting coloured $(\cU, \cV)$-walks is always D-finite (in fact,
algebraic). Proposition~\ref{prop:hadamard} specialises as follows.
\begin{prop}\label{prop:12}
  If $\cS$ is $(1,2)$-Hadamard with $\cS=\left(\cU\times
  \{0\}^2\right) \cup \left(\cV\times \cT\right)$, the
  generating function $O(x,y,z;t)$ of $\cS$-walks 
  confined to the octant is D-finite as soon as 
  the series $C_2(y,z;t)$ that counts $\cT$-walks in the quadrant is D-finite.
\end{prop}

\subsection{The case of $\boldsymbol{(2,1)}$-Hadamard walks and the reflection
  principle}
Assume that the model  $\cS$ is $(2,1)$-Hadamard:
$$ S(x,y,z)= U(x,y)+ V(x,y)T(z). $$
Since the set $\cT$ is (at most) 1-dimensional, the generating function $
C_2(z;t)$ that counts $\cT$-walks in $\ns$ is always D-finite (in fact,
algebraic).

\begin{prop}\label{prop:21}
If $\cS$ is $(2,1)$-Hadamard with $\cS=\left(\cU\times \{0\}\right) \cup
\left(\cV\times \cT\right)$, then $O(x,y,z;t)$ is D-finite as soon as the
series $C_1(x,y,v;t)$ that counts coloured $(\cU, \cV)$-walks in the quadrant
is D-finite.

Moreover, if $T(z)=z+\bar z$, then 
$$ O(x,y,z;t)=[z^{\ge 0}] (1-\bz ^2) Q(x,y,z;t), $$
where $Q(x,y,z;t)$ counts $\cS$-walks confined to $\ns^2\times \zs$.
\end{prop}
\noindent Note that $T(z)$ \emm must be, equal to $(z+\bz)$ for a 3D model.
\begin{proof}
The first statement just paraphrases the second part of
Proposition~\ref{prop:hadamard}. Now assume that $T(z)=z+\bar z$,
and let us apply  the first statement of
Proposition~\ref{prop:hadamard}. The reflection principle for walks on
a line gives 
$$
C_2(z;v)= [z^{\ge 0}] \frac{1-\bz^2}{1-v(z+\bz)} = \sum_{k\ge 0} v^k  [z^{\ge 0}] ({1-\bz^2})(z+\bz)^k.
$$
Let us now write $ C_1(x,y,v;t)= \sum_{k\ge 0}  C_{1,k}(x,y;t)v^k$. Then
\begin{eqnarray*}
O(x,y,z;t)&=& \sum_{k\ge 0}  C_{1,k}(x,y;t) [z^{\ge 0}] ({1-\bz^2})(z+\bz)^k
\\
&=&[z^{\ge 0}] ({1-\bz^2}) \sum_{k\ge 0}  C_{1,k}(x,y;t)(z+\bz)^k
\\
&=&[z^{\ge 0}] ({1-\bz^2}) Q(x,y,z;t)
\end{eqnarray*}
by Observation~\ref{mixt}. This identity can also be obtained by
applying the reflection principle to $\cS$-walks confined to
$\ns^2\times \zs$, but this
proof underlines the connection with Proposition~\ref{prop:hadamard}. 
\end{proof}

\begin{example}\label{ex:hadamard-gessel} 
{\bf A $\boldsymbol{(2,1)}$-Hadamard model.}
Take 
$$
S(x,y,z)= x+ xy+ (\bx+\bx\by)(z+\bz).
$$
This is a $(2,1)$-Hadamard model where $\cU\cap \cV=\emptyset$. The series $
C_1(x,y,v;t) $ counts the so-called Gessel walks by the length ($t$), the
coordinates of the endpoint ($x,y$), and the number of $x^-$ steps ($v$). It
can be expressed in terms of the complete generating function $Q(x,y;t)$ of
Gessel walks, which is known to be algebraic~\cite{BoKa10}. Indeed, $$
C_1(x,y,v;t) =Q(x\sqrt{\bv},y;t\sqrt v)$$ with $\bv=1/v$. The series
$C_2(y;v)$ is the (algebraic) generating function of $\pm 1$-walks in $\ns$,
and $O(x,y,z;t)$ is thus D-finite. The associated group is finite of order 16,
with a zero orbit sum.
\end{example}

\section{Three-dimensional octant models}
\label{sec:ThreeDim}

Starting from the 35\,548 non-equivalent models with at most 6 steps that have
dimension 2 or 3 (Proposition~\ref{prop:interesting}), we first apply the
linear programming technique described at the end of Section~\ref{sec:dim} to
determine their dimension. Doing this using the built-in solver of
Sage~\cite{sage} requires no more than a few minutes of computation time (in
total). We thus obtain 20\,804 truly three-dimensional models.

We analyse them with the tools of Sections~\ref{sec:summary}
to~\ref{sec:hadamard}. Our results are summarised by Table~\ref{table:3D}
displayed in Section~\ref{sec:summary}. On the experimental side, we go
through all models and determine which of them appear to have a finite group.
We find only 170 models with a finite group, of order 8, 12, 16, 24 or 48. The
20\,634 remaining models have a group of order at least 200, which we
conjecture to be infinite. This includes for instance the highly symmetric
model $S=x\by\bz+\bx y\bz+ \bx\by z+xyz$ mentioned in~\cite[Sec.~4.1]{BoKa10}.

Among the 170 models with a finite group, 108 have a non-zero orbit sum. The
algebraic kernel method of Section~\ref{sec:kernel} proves that their complete
generating function is D-finite: up to a permutation of coordinates, the
series $xyzO(x,y,z;t)$ is obtained by extracting the positive part in $z$, $y$
and $x$ in a rational function, as in~\eqref{extracted}. The correctness of
the extraction argument is guaranteed by Condition~\eqref{extraction}, which
holds in each of these 108 cases.

There are 62 models left where the orbit sum is zero. Among these, 43 can be
recognised as Hadamard and have a D-finite generating functions. The remaining
19 are not Hadamard.

\subsection{The Hadamard models with zero orbit sum}
\label{sec:Hzero}

Let us give a few details on these 43 Hadamard models. Among them, 31 are
$(1,2)$-Hadamard and 17 are $(2,1)$-Hadamard (up to a permutation of
coordinates). This means that 5 are both $(1,2)$- and $(2,1)$-Hadamard; this
happens when $\cU=\emptyset$ or $\cV=\emptyset$. For the $(1,2)$-Hadamard
models, we apply Proposition~\ref{prop:12}. The set~$\cT$ is found to be
either Kreweras' model $\{\bone0,0\bone, 11\}$, or its reverse
$\{10,01,\bone\bone\}$ or Gessel's model $\{10,\bone0, 11,\bone\bone\}$. In
all three cases, the generating function $C_2(y,z;v)$ is known to be
algebraic~\cite{gessel-proba,Bous05,BoKa10}, so that $O(x,y,z;t)$ is D-finite
by Proposition~\ref{prop:12}.

For the $(2,1)$-Hadamard models, we apply Proposition~\ref{prop:21}, and thus
need to check if $C_1(x,y,v;t)$ is D-finite. The set $\cU\cup \cV$ is found to
be either Kreweras' model, or its reverse, or Gessel's model (in each case,
possibly with an additional null step $00$). Since there are only three steps
in Kreweras' model, the associated complete generating function $Q(x,y;t)$
records in fact the number of steps of each type (after some algebraic changes
of variables). Moreover, adding a null step and recording its number of
occurrences preserves algebraicity. This implies that in the Kreweras case,
the series $C_1(x,y,v;t)$ defined by~\eqref{C1} is obtained by algebraic
substitutions in $Q(x,y;t)$ and is thus algebraic. The same holds in the
reverse Kreweras case. This leaves us with five models for which $\cU\cup \cV$
is Gessel's model (possibly with a null step). It is known that the associated
complete generating function $Q(x,y;t)$ is algebraic, but does it imply that
$C_1(x,y,v;t)$ is also algebraic? The answer is yes. In all five cases,
$\cU\cap \cV=\emptyset$, so that we just have to check that we can keep track
of the number of $\cV$-steps (see~\eqref{C1}). The five sets $\cV$ are
$\{\bone0,\bone\bone\}$ (this is Example~\ref{ex:hadamard-gessel}),
$\{10,11\}$, $\{01,\bone\bone\}$, $\{0\bone,11\}$ and finally $\{00\}$ (in the
latter case, the model is also $(1,2)$-Hadamard). In each case, it is readily
checked that $C_1(x,y,v;t)$ is obtained by algebraic substitutions in
$Q(x,y;t)$ and is thus algebraic.

\subsection{The non-Hadamard models with zero orbit sum}
\label{sec:nonHzero}

These 19 models, shown in Figure~\ref{fig:19}, remain mysterious. We are
tempted to believe that they are not D-finite, which would mean that the nice
correspondence between a finite group and D-finiteness observed for quadrant
walks does not extend to octant walks. We note that models associated with
isomorphic groups may behave very differently: for instance, if we change the
sign in the $x$-coordinate of each step of the second model of
Figure~\ref{fig:19}, we obtain a model with an isomorphic group
(see~\cite[Lemma~2]{BoMi10}), but a non-zero orbit sum. This new model can be
solved using the kernel method and has a D-finite series.

\begin{figure}[bht]
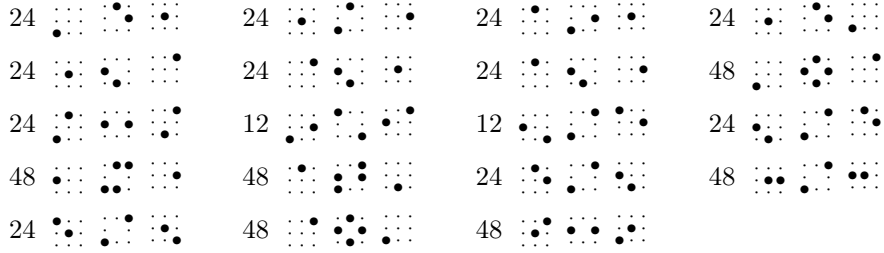

  \centering
  \begin{tabular}{c@{\qquad}c@{\qquad}c@{\qquad}c}
24\ \Stepset10000000000001010000010000 
&24\ \Stepset00001000010000010000001000 
&24\ \Stepset00000001010001000000010000 
&24\ \Stepset00001000000001010100000000 
\\24\ \Stepset00001000001010000000000001 
&24\ \Stepset00000000101010000000010000 
&24\ \Stepset00000001001010000000001000 
&48\ \Stepset10000000001011010000000001 
\\24\ \Stepset10000001000011000010000001 
&12\ \Stepset100 001 000    001 00 100   000 100 001 
&12\ \Stepset00110000010000001000001100 
&24\ \Stepset01010000010000001000001010 
\\48\ \Stepset00010000011000011000001000 
&48\ \Stepset00000001010011001010000000 
&24\ \Stepset00000101010000001010100000 
&48\ \Stepset00001100010000001000110000 
\\24\ \Stepset00001010010000001001010000 
&48\ \Stepset00000000101011010100000000 
&48\ \Stepset00001000100011000100010000 
\end{tabular}  
\caption{The 19 non-Hadamard 3D models with a zero orbit sum, with the
  order of the associated group.  The first one on 
the second line is the natural 3D analogue of Kreweras' model.}
\label{fig:19}
\end{figure}

We have tried hard to guess differential equations for the 19 mysterious
models. In each case, we have calculated the first 5000 terms of the
generating functions $O(x_0,y_0,z_0;t)$ for any choice of $x_0,y_0,z_0$ in
$\{0,1\}$, but we did not find any equation. It could still be that some or
all of the models are D-finite but the equations are so large that 5000 terms
are not enough to detect them. Unfortunately, it is quite hard to compute more
terms, even if the calculations are restricted to a finite field for better
efficiency. The default algorithm based on the recurrence relation~\eqref{rec}
then still takes $O(n^4)$ time and $O(n^3)$ space, which already for $n=5000$
required us to employ a supercomputer (the one we used has 2048 processors and
16000 gigabyte of main memory).

For getting an idea about the expected sizes of equations, we have also
guessed equations for all the 43 Hadamard models with finite group and zero
orbit sum. The formulas given in Section~\ref{sec:hadamard} for their
generating functions give rise to more efficient algorithms for computing
their series expansion. It turns out that the biggest equation appears for the
model
\Stepset00010000110001000000100001, 
for $x=y=z=1$.
It has order~55 and degree~3815, and we needed 20000 terms to construct it.
This suggests that we may not have enough terms in our 19 models to guess
differential equations.

However, we observed that even for the models where $O(1,1,1;t)$ satisfies
only extremely large equations, at least one of the other specialisations
$O(x_0,y_0,z_0;t)$ with $(x_0,y_0,z_0) \in \{0,1\}^3$ satisfies a small
equation that can be recovered from a few hundred terms already. This is in
contrast to the 19 non-Hadamard step sets, where none of these series
satisfies an equation that could be found with 5000 terms. For this reason, we
have some doubts that these models are D-finite.

\section{Two-dimensional octant models}
\label{sec:TwoDim}
We are now left with the 14\,744 two-dimensional models found among the
35\,548 octant models with at most 6 steps
(Proposition~\ref{prop:interesting}). At the moment, we have only studied
their projection on the relevant quadrant:

\begin{quote}
  \emph{For a 2D octant model where the $z$-condition is
    redundant, we  focus on  the series
    $Q(x,y;t):=O(x,y,1;t)$. This series counts quadrant walks with steps in
    $\cS'= \{ij:  ijk \in \cS\}$, which should be thought of as a
    {\emph{multiset}} of steps.}
\end{quote}
For instance, if $\cS$ is the 2D model of Example~\ref{ex:2D}, then
$\cS'=\{0\bone, \bone 1 , \bone 1, 10\}$ contains two copies of the
step $\bone 1$. 

Note that the nature of $Q(x,y;t)$ is not affected by adding a step $00$ to
$\cS'$: it corresponds to substituting $t \mapsto t/(1-t)$ in the generating
function. This is why we do not consider models with a null step.

\subsection{Quadrant models obtained by projection}
We obtain by projection  527 models of
cardinality at  most 6 (after deleting the 00 step, and
identifying models that only differ by an $xy$ symmetry). 
Those with no repeated 
step were  previously classified by Bousquet-M\'elou and
Mishna~\cite{BoMi10}, but we also have models with repeated steps.

\begin{table}
\def\algebraic{\leaf{\fbox{\textbf{algebraic}}}}
\def\dfinite{\leaf{\fbox{\textbf{D-finite}}}}
\def\cdfinite{\leaf{\fbox{\textbf{D-finite?}}}}
\def\notdfinite{\leaf{\fbox{\textbf{not D-finite?}}}}
\edgeheight=5pt\nodeskip=0pt\leavevmode\kern-1.2em
\tree{\begin{tabular}{c}
2D quadrant models with $\leq6$ steps obtained by projection\\
$527=[7,41,141,338]$
\end{tabular}}
{
      \tree{\begin{tabular}{c}
$|G|{<}\infty$\\
$118=[5,19,35,59]$
\end{tabular}}{
        \tree{\begin{tabular}{c}
OS${\neq}0$\\
$95=[3,14,29,49]$
\end{tabular}}
{\tree{\begin{tabular}{c}kernel method\\
$94=[3,14,29,48]$\end{tabular}}
{\dfinite}\tree{
\begin{tabular}{c}
model $ \cS_1$\\
extraction fails,\\
computer alg.
\end{tabular}}
{\dfinite}}
        \tree{\begin{tabular}{c}
OS${=}0$\\
$23=[2,5,6,10]$
\end{tabular}}
{\tree{\begin{tabular}{c}
related to a \\
model of Fig.~\ref{fig:alg}\\
$22=[2,5,6,9]$
\end{tabular}}
{\algebraic}
\tree{\begin{tabular}{c}
model $\overline \cS_1$\\
(Fig.~\ref{fig:quatre})\\
computer alg.
\end{tabular}}{\algebraic}
      }}
      \tree{\begin{tabular}{c}
$|G|{=}\infty$ ?\\
$409=[2,22,106,279]$
\end{tabular}}{\notdfinite}
    }
\caption{Summary of our results and conjectures for the 527 quadrant models
  with at most 6 steps obtained by projection of a 2D octant model. The numbers in brackets give for
  each class the number of models of cardinality 3, 4, 5 and 6.}
\label{table:2Drepeated}
\end{table}

As in the 3D case, we first try to determine, experimentally, when the
group $G$ is finite and when the specialisations of $Q(x,y;t)$
obtained when $\{x,y\}\subset \{0,1\}$ are D-finite.
Our results are summarised in Table~\ref{table:2Drepeated}. 

We find 118 models with a finite group, of order at most 8 (more precisely, of
order 4, 6 or 8). The remaining ones have a group of order at least 200, which
we conjecture to be infinite. (This is known for models with no repeated
step~\cite{BoMi10}, and it is likely that the methods of~\cite{BoMi10} can
prove the other cases as well.) We observe that for each model for which
$Q(1,1;t)$ is guessed D-finite, the group is finite.

Among the 118 models with a finite group, 95 have a non-zero orbit sum. For 94
of them, the algebraic kernel method establishes the D-finiteness of
$Q(x,y;t)$. For 92 of these 94 models, the extraction of the positive part in
$x$ and $y$ is justified by the fact that every element $[x',y']$ of the
orbit, other than $[x,y]$, satisfies the 2D counterpart of~\eqref{extraction},
namely
\begin{equation}\label{extraction2D}
\{x',y'\} \subset \qs(x)[\by] \quad \hbox{or } \quad \{x',y'\} \subset \qs[\bx,
y]
\end{equation}
(up to a permutation of coordinates). This property \emm does not hold, for
the remaining two models $\cS_0$ and $\overline \cS_0$, shown in
Figure~\ref{fig:quatre}, but the extraction is still valid, as detailed in
Section~\ref{sec:kernel-interesting}. The 95th model with a finite group and a
non-zero orbit sum is the model $\cS_1$ also shown in Figure~\ref{fig:quatre}.
As will be detailed in Section~\ref{sec:p1}, it shares with models with a zero
orbit sum the property that the orbit equation does \emm not, characterise the
series $Q(x,y;t)$. We will prove by computer algebra that this series is still
D-finite, but transcendental (see Section~\ref{sec:p1}).

\begin{figure}[ht]
\begin{center}
\hskip -10mm\scalebox{0.9}{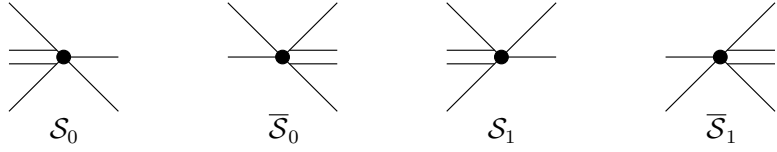}
\caption{Four interesting quadrant models with repeated steps, only
  differing by a symmetry of the square. All are D-finite, and
  $\overline \cS_1$ is even algebraic. See
  Section~\ref{sec:kernel-interesting} for $\cS_0$ 
and $\overline \cS_0$, Section~\ref{sec:p2} for $\overline\cS_1$ and
Section~\ref{sec:p1} for $\cS_1$. Note that with only one copy of the
repeated step, none of these models would be D-finite~\cite{BoRaSa12}. } 
\label{fig:quatre}
\end{center}
\end{figure}

We are now left with the 23 models having a zero orbit sum. Exactly 22 of them
are obtained from the four algebraic quadrant models of Figure~\ref{fig:alg}
by possibly repeating some steps. In each of these 22 cases, the generating
function of the models with multiple steps can be written as an algebraic
substitution in the complete generating function of the model with no
repeated steps. This is obvious in the Kreweras or reverse Kreweras case,
since their complete generating function actually keeps track of the number
of occurrences of each step. This is also obvious in the double Kreweras case:
it has already 6 steps, so that it only occurs in our list with non-repeated
steps. In the Gessel case, the steps are at most duplicated, and the set $\cV$
of duplicated steps is $\{\bone 0, \bone\bone\}$, $\{10,11\}$,
$\{01,\bone\bone\}$, or $\{0\bone,11\}$. This is the same list as that of
Section~\ref{sec:Hzero}, and in each case, the complete generating function
of the model with repetitions is obtained by an algebraic substitution in the
(algebraic) generating function of Gessel's walks.

The 23rd model with a finite group and zero orbit sum is the model $\overline
\cS_1$ of Figure~\ref{fig:quatre}. Using the technique applied to Gessel's
model in~\cite{BoKa10}, we will prove that its generating function
$Q(x,y;t)$ is algebraic (Section~\ref{sec:p2}). Moreover, the series
$Q(0,y;t)$ has nice hypergeometric coefficients.

\subsection{Two interesting applications of the kernel method}
\label{sec:kernel-interesting}
We discuss here the two models for which the orbit sum is non-zero,
and where the kernel method works even though
Condition~\eqref{extraction2D} does not hold. These models are
$\cS_0=\{\bone \bone, \bone 0, \bone 0, \bone 1, 10, 1\bone\}$ (with a
repeated west step) and its
reverse $\overline{\cS}_0=\{11,10,10,1\bone, \bone 0, \bone 1\}$
(Figure~\ref{fig:quatre}). In both 
cases, the group $G$ has order 6. 

For $\cS_0$, the non-trivial elements in the orbit of $[x,y]$ are
\begin{multline*}
\quad [\bx(1+y),y],  \quad [\bx(1+y),\by +\bx^2\by(1+y)^2],
\quad [\bx +(x+\bx )\by,\by +\bx^2\by(1+y)^2], \\
 \quad [\bx +(x+\bx )\by,\by(1+x^2)], \quad
[x,\by(1+x^2)]  .
\end{multline*}
All coordinates are Laurent polynomials in $x$ and $y$. Moreover, all elements
$[x',y']$ satisfy~\eqref{extraction2D}, with the exception of the third.
Still, all monomials $x^i y^j$ occurring in this element satisfy $i+j\le 0$.
Hence this is also true for the monomials occurring in $x'y' Q(x',y')$, so
that this series does not contribute when one extracts the positive part (in
$x$ and $y$) of the orbit sum. One thus obtains
$$
xyQ(x,y)=[x^{>0}][y^{>0}] \frac{(x-\bx-x\by-\bx\by) (x-\bx-\bx y) (\bx
  y-x\by-\bx\by)}{1-t(1+\by)(x+\bx(1+y))}.
$$
The coefficient extraction can be performed explicitly and gives rise to nice
numbers. Details of the extraction are left to the reader.

\begin{prop}
For the model $\cS_0=\{\bone \bone, \bone 0, \bone 0, \bone 1, 10, 1\bone\}$,
the number of walks of length $n$ ending at $(i,j)$ is zero unless $n$ can be
written as $2m+i$, in which case
$$
q(i,j;n)= \frac{(i+1)(j+1)\left[(2i+3j+6)m+i^2+5i+2ij+3j+6\right]\, n!
(3m+i+2)!}
{m!(m+i)!(m-j)!(2m+i+j+3)!(m+i+1)(m+i+2)(m+1)}.
$$
\end{prop}
\noindent
{\bf Remark.} The asymptotic behaviour of $q(0,0;2m)$, which is of the form $c\,
27^mn^{-4}$, prevents the series $Q(0,0;t)$ from being algebraic~\cite{flajolet-context-free}. 

\medskip

Let us now address the model with reversed steps,
$\overline{\cS}_0=\{11,10,10,1\bone ,\bone 0 ,\bone 1\}$. Now the orbit
of $[x,y]$ involves rational functions which are not Laurent
polynomials. Besides $[x,y]$ itself, the orbit contains the following five
pairs:
\begin{multline*}
\left[\frac{\bx y}{1+y} ,y \right],\quad \left[\frac{\bx y}{1+y} ,  \frac{\bx^2 y}{(1+y)^2+\bx^2
  y^2}\right],\quad
\left[\frac{\bx}{1+y(1+\bx^2)} ,  \frac{\bx^2 y}{(1+y)^2+\bx^2
     y^2}\right],
\\
\left[\frac{\bx}{1+y(1+\bx^2)} , \frac{x^2\by}{1+x^2} \right], \quad
  \left[x, \frac{x^2\by}{1+x^2} \right].
\end{multline*}
Let $[x',y']$ be one of these pairs, $i,j$ two non-negative integers,
and let us expand the series $x'^i y'^j$ first in $y$, and then in
$x$. We claim that the resulting series (a Laurent series in $y$ whose
coefficients are Laurent series in $x$) does not contain any monomial
$x^a y^b$ with $a$ and $b$ positive. This is obvious for all pairs,
except the fourth one, for which we have to perform the following
calculation: for $i>0$,
$$
  x'^i y'^j= \sum_{n\ge 0} (-1)^n{n+i-1 \choose i-1} y^{n-j}\bx ^i (1+\bx^2)^{n-j}. 
$$
 This expansion shows that if the exponent of $y$ is positive ($n>j$),
 then the exponent of $x$ is negative. Hence the only series that
 contributes when expanding the orbit equation in $y$ and $x$ is $xyQ(x,y)$
 and finally
$$
xyQ(x,y)= [x^{>0}][y^{>0}] \frac{y(x-\bx+xy+\bx y) (x-\bx +x\by)(xy+\bx y-x\by)}{(1+x^2)\left( (1+y)^2+\bx^2 y^2\right)\left(1-t(1+y)(\bx+x(1+\by))\right)}.
$$
As opposed to the previous example, the numbers $q(i,j;n)$ do not
appear to be hypergeometric in $i$, $j$ and $n$, but only in $i$ and
$n$, for $j$ fixed. For instance,
\begin{align*}  q(i,0;2m+i)&=\frac{(i+1)(i+2)(2m+i)!(3m+2i+3)!}{m!(m+i)!(m+i+2)!(2m+i+2)!(m+i+1)(2m+2i+3)}.
\end{align*}
Of course, the series $Q(0,0;t)$, being the same as for the model
$\cS_0$ studied previously, is transcendental.

\subsection{Algebraicity}

We have also tried to guess -- and then prove -- algebraic equations for the
specialisations $Q(x_0,y_0;t)$ with $\{x_0, y_0\}\subset\{0,1\}$. Here is a
summary of our results.

First, we have of course the 23 models with finite group and zero orbit sum,
for which the complete generating function $Q(x,y;t)$ is algebraic over
$\qs(x,y,t)$ (see Table~\ref{table:2Drepeated}). This explains for instance
why the octant models with characteristic polynomials $\bz + \bx(\by+1+y) +
xyz$ or $ \bx\bz(\by+1+y) +xy+z$ were conjectured to have an algebraic length
generating function $O(1,1,1;t)$ in~\cite[Table~2]{BoKa09}: both are
2-dimensional, with the $y$-condition redundant, and their projections on the
$xz$-plane are $\{0\bone, \bone 0, \bone 0, \bone 0, 11\}$ and $\{\bone\bone,
\bone\bone,\bone\bone, 10, 01\}$. These are respectively versions of Kreweras'
model and of its reverse, with multiple steps, and they are algebraic.

For models with a finite group and a \emm non-zero, orbit sum, we conjecture
that the complete generating function $Q(x,y;t)$ is transcendental. However,
we have also found experimentally, and then proved, a few algebraic
specialisations of these series. This is for instance the case for the
specialisation $Q(0,0;t$) of the model $\cS_1$ of Figure~\ref{fig:quatre}:
this series is the same as for the model $\overline \cS_1$, which is
algebraic. The other models with algebraic specialisations are $\{0\bone , 10,
\bone 1\}$ (possibly with repeated steps) and $\{0\bone , 10, \bone 1, \bone
0, 01, 1\bone\}$. Both models are known to have certain algebraic
specialisations~\cite[Section~5.2]{BoMi10}. Our algebraic guesses are all
consequences of the following two results.

\begin{itemize}
 \item For $\cS= \{0\bone , 10, \bone 1\}$, the generating function $Q(1,1;t)$ is
   algebraic~\cite[Prop.~9]{BoMi10}. With the same proof, one shows
     that this remains true with a  weight $\delta$ for diagonal steps
     $\bone 1$.
\item For $\cS=\{0\bone , 10, \bone 1, \bone 0, 01, 1\bone\}$, the generating function $Q(1,1;t)$ is
   algebraic~\cite[Prop.~10]{BoMi10}. One step of the proof is to
   prove that $xQ(x,0)+\bx(1+\bx)Q(0,\bx)$ is algebraic. Setting
   $x=1$, and using the $xy$-symmetry, shows that  $Q(0,1;t)$
   is also algebraic. 
 \end{itemize}
\begin{example}
Consider  the octant model
$ \cS = \{0\bar 10 , 101, \bar 1 1\bar 1 , \bar 1 10, \bar 1 11
\}$, which was conjectured to have an algebraic length
generating function $O(1,1,1;t)$ in~\cite[Table~2]{BoKa09} (up to a permutation of
$x$ and $y$, it is the first model in that table).
The $z$-condition is easily seen to be redundant, and $\cS$ is
two-dimensional. Let $\cU$ be the set obtained by projecting
$\cS$ on the $xy$-plane and deleting repeated steps: $\cU = \{0\bar 1 , 10, \bar 1 1 
\}$. Let $Q(x,y;t) $ be the quadrant series associated with
$\cU$. Denoting by $a(w)$, $b(w)$ and $c(w)$ the number of occurrences
of $0\bar 1 , 10$ and $ \bar 1 1$ in a $\cU$-walk $w$, we have
$$
Q(X,Y;T)= \sum_w (\bY T)^{a(w)}(X T)^{b(w)} (\bX Y T)^{c(w)}.
$$
Moreover, by lifting every $\cU$-walk $w$ into all $\cS$-walks that
project on it, we obtain
$$
O(x,y,z;t)=\sum_w (\by t)^{a(w)} (xzt)^{b(w)} \left( \bx y (1+z+\bz)
  t\right)^{c(w)}.
$$
Comparing these two identities shows that
\[ O(x,y,z;t) = Q\left( \frac{xz}{Z^{1/3}},yZ^{1/3};tZ^{1/3}\right) \quad \text{ where $Z=1+z+z^2$}.\]
Since $Q(x,y;t)$ is D-finite we have that $O(x,y,z;t)$ is D-finite.
Furthermore, $Q\left(x,\frac{1}{x};t\right)$ is
algebraic~\cite[Prop.~9]{BoMi10},  so
$O\left(x,y,\frac{1}{xy};t\right)$, and specifically $O(1,1,1;t)$, is
also algebraic. 
\end{example}

\section{Computer algebra solutions}
\label{sec:computer}
As discussed in the previous section, there are  two D-finite quadrant
models (with repeated steps) for which we have only found proofs based
on computer algebra. We now  give details on these proofs. The two
models in question are $\cS_1=\{11, 10, \bar11, \bar10, \bar10,
\bar1\bar1\}$ and its reverse $\overline\cS_1=\{\bar1\bar1, \bar10, 1\bar1, 10,
10, 11\}$ (Figure~\ref{fig:quatre}). We will show that $\cS_1$ is
D-finite (but transcendental) and
 $\overline\cS_1$  algebraic. We begin with the latter model.

\subsection{The model $\boldsymbol{\overline \cS_1=\{ \bar 1  \bar 1 ,
    \bar 1  0,1 \bar 1 , 1    0, 1 0, 1 1\}}$ is algebraic}
\label{sec:p2}
The group of this model is finite  of order 6. The orbit sum is zero,
which implies   that the orbit  equation does not characterise
$Q(x,y)$ (in fact,   $Q(x,y)=1$ is another solution).  
We first establish the algebraicity of the series $Q(0,0;t)$,
which  counts quadrant walks ending at the origin, 
often called \emm excursions,.

\begin{lemma}\label{lemma:new00}
  For the quadrant model  $\overline\cS_1=\{ \bar 1  \bar 1 , \bar 1  0,1 \bar 1 , 1
  0, 1 0, 1 1\}$, the generating function $Q(0,0;t)$ is
\begin{equation}\label{Q00-alg}
Q(0,0;t)= \sum_{n\ge 0} \frac{6(6n+1)!(2n+1)!}{(3n)!(4n+3)!(n+1)!} t^{2n}.
\end{equation}
It is  algebraic of degree $6$. Denoting $Q_{00}\equiv Q(0,0;t)$, we have 
\begin{alignat}
1
&16 t^{10} Q_{00}^6+48 t^8 Q_{00}^5+8 (6 t^2+7) t^6 Q_{00}^4+32 (3
t^2+1) t^4 Q_{00}^3 
\nonumber\\
&{}+(48 t^4-8 t^2+9) t^2 Q_{00}^2+(48 t^4-56 t^2+1)Q_{00}+(16 t^4+44 t^2-1) = 0.
\label{alg-Q00}
\end{alignat}
A parametric expression of $Q(0,0;t)$ is
$$
t^2Q(0,0;t)= Z(1-6Z+4Z^2),
$$
where $Z\equiv Z(t)$ is the unique power series in $t$ with constant
term $0$ satisfying
$$
Z(1-Z)(1-2Z)^4= {t^2}.
$$
\end{lemma}
\begin{proof}
We prove the first claim~\eqref{Q00-alg} by computer algebra, using the method
applied in~\cite{KaKoZe08} for proving that the number of Gessel excursions is
hypergeometric --- the so-called \emm quasi-holonomic ansatz, (see
also~\cite{kauers07v}). As the details of this approach are already explained
in this earlier article, we give here only a rough sketch of our computations.

First compute the coefficients $q(i,j;n)$ of $Q(x,y;t)$ for $0\leq i,j,n\leq
50$, using the quadrant counterpart of the recurrence relation~\eqref{rec}.
Then, use this data to guess~\cite{kauers09a} a system of multivariate linear
recurrence equations with polynomial coefficients (in $i,j$ and $n$) satisfied
by the numbers $q(i,j;n)$. Then use the algorithm from~\cite{kauers07v} to
prove that these guessed recurrences indeed hold. Then use linear algebra (or
a variant of Takayama's algorithm~\cite{takayama90}, as was done
in~\cite{KaKoZe08}) to combine these recurrences into a linear recurrence of
the following form: for all $i,j$ and $n$,
$$
\sum_{a,b,c ,i',j',n'\ge 0} \kappa(a,b,c,i',j',n')\, i^a j^b n^c
 \, q(i+i',j+j';n+n')=0,
$$ 
where $\kappa(a,b,c,i',j',n')\in\qs$ is non-zero for only finitely many terms,
and $\max(a,b)>0$ whenever $\max(i',j')>0$. Upon setting $i=j=0$, this
recurrence gives a recurrence in $n$ for the coefficients $q(0,0;n)$ of
$Q(0,0;t)$. We found in this way a recurrence of order~$8$ and degree~$14$. It
is then routine to check that the coefficients of~\eqref{Q00-alg} satisfy this
recurrence and that the right number of initial values match.

Now, the polynomial equation~\eqref{alg-Q00} has clearly a unique power series
solution (think of extracting by induction the coefficient of $t^n$). Using
algorithms for holonomic functions~\cite{gfun,mallinger96}, one can construct
a linear differential equation satisfied by this solution, and then a
recurrence relation satisfied by its coefficients. It suffices then to check
that the coefficients of~\eqref{Q00-alg} satisfy this recurrence and the right
number of initial conditions.

The parametric form of the solution can be found with the
{\tt algcurves} package of Maple. Using this parametrisation, one can
also recover the expansion~\eqref{Q00-alg} of $Q(0,0;t)$ using the Lagrange
inversion formula.
\end{proof}

We now extend Lemma~\ref{lemma:new00} to the algebraicity of the complete
series $Q(x,y;t)$.

\begin{prop}\label{prop:quadrant-alg}
 For the quadrant model  $\overline\cS_1=\{ \bar 1  \bar 1 , \bar 1  0,1 \bar 1 , 1
 0, 1 0, 1 1\}$ (with an east repeated step), the generating function $Q(x,y;t)\equiv
 Q(x,y)$ is 
 algebraic of degree~$12$.
 It satisfies
\begin{equation}\label{Qxy-sol}
Q(x,y)=\frac{xy-t(1+x^2)Q(x,0)-t(1+y)Q(0,y)+tQ(0,0)}{xy \left( 1-t(1+\by)(\bx+x(1+y))\right)}.
\end{equation}
The specialisations $Q(x,0;t)$ and 
$Q(0,y;t)$ can be written in parametric form as follows.
Let $T\equiv T(t)$ be the unique power series in $t$ with constant
 term $0$ such that
$$
T(1-4T^2)=t.
$$
Let $S\equiv S(t)$ be the unique power series in $t$ with constant term $0$ such
that
$$
S=T(1+S^2).
$$
Equivalently,  $S(1-S^2)^2=t(1+S^2)^3$. Then $Q(x,0;t)$ has degree $12$ and is quadratic over $\qs(x,S)$:
\begin{multline}\label{Qx0-sol}
  Q(x,0;t)=\left( \frac{1+S^2}{1-S^2}\right)^3  \times \\
\frac{x(1+6S^2+S^4)-2S(1-S^2)(1+x^2)-(x-2S+xS^2)\sqrt{(1-S^2)^2-4xS(1+S^2)}}{2x(1+x^2)S^2}.
\end{multline}
Let also $W\equiv W(y;t)$ be the unique power series in $t$ with constant
 term $0$ such that
$$
W \left({1-(1+y)W}\right)=  T^2.
$$
Then $Q(0,y;t)$ has degree $6$ and is rational in $T$ and $W$:
\begin{equation}\label{Q0y-sol}
Q(0,y;t)= t^{-2} W (1-4T^2-2W).
\end{equation}
Moreover, its coefficients are doubly hypergeometric:
$$
Q(0,y;t)= \sum_{n\ge j\ge 0}
\frac{6(2j+1)!(6n+1)!(2n+j+1)!}{j!^2(3n)!(4n+2j+3)!(n-j)!(n+1)}
y^jt^{2n}.
$$
\end{prop}
\begin{proof}
We prove the claims again by computer algebra, this time using the technique
applied in~\cite{BoKa10} for proving that Gessel's model is algebraic. Again,
as there is nothing new about the logical structure of the argument, we give
here only a summary of our computations (and complete a missing argument
in~\cite{BoKa10}).

The functional equation defining $Q(x,y)$ reads
\begin{multline*}
  xy \left( 1-t(1+\by)(\bx+x(1+y))\right)
Q(x,y)=\\
xy-t(1+x^2)Q(x,0)-t(1+y)Q(0,y)+tQ(0,0).
\end{multline*}
This is of course equivalent to~\eqref{Qxy-sol}. Recall that $Q(0,0)$
is known from Lemma~\ref{lemma:new00}. The (standard) kernel
method~\cite{BoPe03,Bous05} consists in cancelling the kernel by an appropriate
choice of $x$ or $y$ to obtain equations relating the series $Q(x,0)$
and $Q(0,y)$. Applied to the above equation, it gives:
\begin{equation}\label{pair-eq}
\left\{
\begin{array}{ccc}
  (1+x^2)Q(x,0)+(1+Y)Q(0,Y)&=&{xY}/t +Q(0,0),\\
(1+X^2)Q(X,0)+(1+y)Q(0,y)&=&{Xy}/t +Q(0,0),
\end{array}
\right.
\end{equation}
where  $X$ and $Y$ are power series in $t$ with Laurent coefficients
in $y$ and $x$, respectively, defined by
$$
\begin{array}{llllll}
   X\equiv X(y;t)&=&\displaystyle \frac{1-\sqrt{1-4t^2\by^2(1+y)^3}}{2t\by(1+y)^2}
\\
\\
&=&t(1+\by)+\by^3(1+y)^4t^3+ O(t^5) & \in \qs[y,\by][[t]],
\\
\\
Y\equiv Y(x;t)& =&\displaystyle\frac{1-t\bx-2tx-\sqrt{1-2t\bx-4tx+t^2\bx^2}}{2tx}
\\
\\
&=&(\bx+x)t+(\bx+x)(\bx+2x)t^2+ O(t^3)& \in \qs[x,\bx][[t]].
\end{array}
$$
  It follows from~\eqref{pair-eq} and~\cite[Lemma~7]{BoKa10}
 that $(Q(x,0),Q(0,y))$ is in fact the unique  pair of formal power series
 $(F,G)\in\qs[[x,t]]\times \set   Q[[y,t]]$  satisfying the functional
 equations 
 \begin{equation}\label{eq:7}
\left\{
\begin{array}{ccc}
  (1+x^2)F(x)+(1+Y)G(Y)&=&{xY}/t +Q(0,0),\\
(1+X^2)F(X)+(1+y)G(y)&=&{Xy}/t +Q(0,0).
\end{array}
\right.
\end{equation}
Let us now define $F$ and $G$ by the right-hand sides of~\eqref{Qx0-sol}
and~\eqref{Q0y-sol} respectively (these expressions have been guessed from the
first coefficients of $Q(x,y)$). We claim that $F$ and $G$ belong to
$\qs[x][[t]]$ and $\qs[y][[t]]$ respectively. This is obvious for~$G$, since
$T$ is a series in $\qs[[t]]$ and $W$ belongs to $ t^2\qs[y][[t]]$. Let us now
consider the series $F$, that is, the right-hand side of~\eqref{Qx0-sol}, and
ignore the initial factor, which is clearly in $\qs[[t]]$. Then the numerator
is a power series in $S$ with polynomial coefficients in $x$. It is readily
checked that it vanishes at $x=0$ and $x=\pm i$, so that these coefficients
are always multiples of $x(1+x^2)$. Moreover, the numerator is $O(S^2)$, so
that we can conclude that $F$ belongs to $\qs[x][[S]]$, and hence to
$\qs[x][[t]]$.

We want to prove that $F$ and $G$ satisfy the system~\eqref{eq:7}. Since $X,
Y, F, G$ and $Q(0,0)$ are algebraic, this can be done by taking resultants,
and then checking an appropriate number of initial values.

This concludes the proof of~\eqref{Qx0-sol} and~\eqref{Q0y-sol}. The fact that
$Q(x,y)$ has degree $12$ is obtained by elimination (it has degree 4 over
$\qs(x,y,T)$). The (computational) proof of the hypergeometric series
expansion of $Q(0,y)$ is similar to the proof of the expansion of $Q(0,0)$ in
Lemma~\ref{lemma:new00}: the key step is to derive from the algebraic equation
satisfied by $Q(0,y)$ a differential equation (in $t$) for this series, and
then a recurrence relation for the coefficient of $t^n$.
\end{proof}

\subsection{The model $\boldsymbol{\cS_1=\{\bar 1 \bar 1 , \bar 1 1 ,
\bar 1 0, \bar 1 0, 1 0,1 1\}}$ is D-finite}
\label{sec:p1}

 The kernel of this model is
 $$
K(x,y)=  1-t(1+y)\left(\bx(1+\by)+x\right).
$$
The functional equation defining $Q(x,y)$ reads:
$$
xyK(x,y)Q(x,y)= xy-tQ(x,0)-t(1+y)^2Q(0,y)+tQ(0,0).
$$
Again, the group has order 6. The orbit  equation reads:
 \begin{multline*}
   xyQ(x,y) - \bx(1+y)Q(\bx(1+\by),y )
 + \frac{x(1+y)}{(1+y)^2+x^2y^2}\, Q\left(\bx(1+\by),
   \frac{x^2y}{(1+y)^2+x^2y^2}\right) 
 \\
 - \frac{xy(1+y+x^2y)}{(1+y)^2+x^2y^2}\, Q\left(\bx(1+y)+xy,
   \frac{x^2y}{(1+y)^2+x^2y^2}\right)\\
 +\frac{\bx\by(1+y+x^2y)}{1+x^2} Q\left(\bx(1+y)+xy,
   \frac{\by}{1+x^2}\right)
 -\frac{x\by}{1+x^2}\, Q\left(x,
   \frac{\by}{1+x^2}\right)
 \\=
 \frac{\left(1+y(1-x^2)\right) \left(1-y^2(1+x^2)\right) \left( 1-x^2 +y(1+x^2)\right)}{xy(1+x^2)K(x,y)\left((1+y)^2+x^2y^2\right)}.
 \end{multline*}
  The right-hand side is non-zero, but this equation does not define $Q(x,y;t)$ uniquely in the ring $\qs[x,y][[t]]$. In fact, the associated
 \emm homogeneous, equation (in $Q(x,y)$) seems to have an infinite dimensional space of
 solutions. It  includes at least the following polynomials in $x$ and $y$:
 $$
 x, \quad  2xy+{x}^{3}y , \quad {x}^{2}y+{x}^{2}+y+2 ,
  \quad
  {x}^{3}{y}^{2}-{x}^{3}y+{x}^{3}+2x{y}^{2}
 . $$
We show in this subsection that $Q(x,y;t)$ is D-finite in $x,y$ and $t$, but
transcendental. To our knowledge, it is the first time that the D-finiteness
of a (non-algebraic) quadrant model is proved via computer algebra.

\begin{prop}\label{prop:S1-D-finite}
  The complete generating function $Q(x,y;t)$ for the model $\cS_1$
  is D-finite in its three variables.
\end{prop}
\begin{proof}
Saying that $Q(x,y;t)$ is D-finite means that the $\set Q(x,y,t)$-vector space
generated by $Q(x,y;t)$ and all its derivatives $D_x^i D_y^j D_t^k\,Q(x,y;t)$
(for $i,j,k \in\set N$) has finite dimension. This vector space is isomorphic
to the algebra $\set Q(x,y,t)[D_x,D_y,D_t]/\mathfrak{I}$, where $\set
Q(x,y,t)[D_x,D_y,D_t]$ is the Ore algebra of linear partial differential
operators with rational function coefficients, and $\mathfrak{I}$ is the left
ideal consisting of all the operators in this algebra which map $Q(x,y;t)$ to
zero. The condition that $Q(x,y;t)$ be D-finite is thus equivalent to
$\mathfrak{I}$ having Hilbert dimension~0.

The structure of the argument is similar to the proof of
Proposition~\ref{prop:quadrant-alg}, but now with differential equations
instead of polynomial equations. Again, it suffices to prove that the
specialisations $Q(x,0)$ and $Q(0,y)$ are D-finite.

First, we calculate the coefficients $q(i,j;n)$ of $Q(x,y;t)$ for $0\leq
i,j,n\leq 100$ using the combinatorial recurrence relation. From these, two
systems of differential operators annihilating $Q(x,0;t)$ and $Q(0,y;t)$,
respectively, can be guessed. The ideals $\mathfrak{I}_{x}$ and
$\mathfrak{I}_{y}$ generated by the guessed operators have Hilbert
dimension~0, as can be checked by a Gr\"obner basis computation. We used
Koutschan's package~\cite{koutschan10c} for this and all the following
Gr\"obner basis computations. The generators of $\mathfrak{I}_{x}$ and
$\mathfrak{I}_{y}$ are somewhat too lengthy to be reproduced here. They are
available on \href{http://www.risc.jku.at/people/mkauers/bobokame}{Manuel
Kauers' website}. We simply report that $\dim_{\set Q(x,t)}\set
Q(x,t)[D_x,D_t]/\mathfrak{I}_{x}= \dim_{\set Q(y,t)}\set
Q(y,t)[D_y,D_t]/\mathfrak{I}_{y}= 11$, and that the Gr\"obner bases for
$\mathfrak{I}_{x}$ and $\mathfrak{I}_{y}$ with respect to a degree order are
roughly 1Mb long.

The proof consists in showing that these guesses are correct. As in the
previous proof, we will first exhibit a pair of series $(F\equiv F(x;t),
G\equiv G(y;t))$ in $\qs[[x,t]] \times\qs[[y,t]]$ that satisfy the guessed
differential equations and certain initial conditions. Then, we will prove
that they satisfy the system
\begin{equation}\label{eq:7-bis}
\left\{
\begin{array}{ccc}
  F(x)+(1+Y)^2G(Y)=&{xY}/t +Q(0,0),\\
F(X)+(1+y)^2G(y)=&{Xy}/t +Q(0,0),
\end{array}
\right.
\end{equation}
which, by the standard kernel method,  characterises the pair
$(Q(x,0),Q(0,y)) $ in $\qs[[x,t]]\times \qs[[y,t]]$. Here, the
algebraic series $X$ and $Y$ that cancel the kernel are given by
$$
\begin{array}{llllll}
   X\equiv X(y;t)&=&\displaystyle \frac{1-\sqrt{1-4t^2\by(1+y)^3}}{2t(1+y)}
\\
\\
&=&t\by(1+y)^2+\by^2(1+y)^5t^3+ O(t^5) & \in \qs[y,\by][[t]],
\\
\\
Y\equiv Y(x;t)& =&\displaystyle\frac{1-tx-2t\bx-\sqrt{1-2tx-4t\bx+t^2x^2}}{2t(x+\bx)}
\\
\\
&=&\bx t+(2\bx^2+1)t^2+ O(t^3)& \in \qs[x,\bx][[t]].
\end{array}
$$
  To exhibit  a power series of $\set Q[[x,t]]$ 
  annihilated by $\mathfrak{I}_{x}$, we construct a nonzero element of
  $\mathfrak{I}_{x}\cap\set Q(x,t)[D_t]$, free of~$D_x$. For such an operator
  it is easy to check that it admits a solution in $\set
  Q(x)[[t]]$. To determine this series uniquely, we prescribe the
  first coefficients of its expansion in $t$ to  coincide
  with those of   $Q(x,0)$.  It remains to  prove that this series,
  henceforth denoted $F\equiv F(x;t)$, actually
  belongs to  $\qs[x][[t]]$. In order to do so,  we look at  the
  recurrence relation satisfied by its coefficients (this recurrence
  can be derived from the  
  differential equation in $t$). Alas, it reads
$$
P(n,x) f_n(x)= \sum_i P_i(n,x) f_{n-i}(x), 
$$
where $P$ and the $P_i$'s are polynomials, and $P$ actually involves $x$ (we
denote by $f_n(x)$ the coefficient of $t^n$ in the series $F(x;t)$). However,
once we apply a desingularisation algorithm to this
recurrence~\cite{abramov99b}, the leading coefficient becomes free of~$x$.
Since the first values of $f_n(x)$ are polynomials in $x$, this implies that
our series $F(x;t)$ lies in $\set Q[x][[t]]\subseteq\set Q[[x,t]]$. In the
same way we construct a series $G(y;t)$ in $\set Q[[y,t]]$ that is cancelled
by $\mathfrak{I}_{y}$ and coincides with $Q(0,y;t)$ far enough to be uniquely
determined by these two conditions.

It remains to prove that the pair $(F,G)$ satisfies~\eqref{eq:7-bis}. Using
algorithms for D-finite closure properties, together with the defining
algebraic equations for $X$, $Y$ and $Q(0,0)$ (Lemma~\ref{lemma:new00}), we
can compute differential equations in $t$ for the left-hand side and the
right-hand side of these identities (which are both power series in $t$, with
coefficients in $\qs[y,\by]$ or $\qs[x, \bx]$, respectively). After checking
an appropriate number of initial terms, it follows that $F$ and $G$ form a
solution of~\eqref{eq:7-bis}, and this concludes our proof.
\end{proof}

The next result proves that the D-finite generating function $Q(x,y;t)$ for
the model $\cS_1$ is transcendental. We have not been able to find a direct
proof of this; arguments based on asymptotics for coefficients of several
specialisations are not enough, since for instance, the coefficients of
$Q(1,0;t)$ appear to grow like $c\,(2+\sqrt{3})^n/n^{3/2}$, for some
constant~$c$, and those of $Q(0,1;t)$ appear to grow like $c\,(4 \sqrt{2})^n/
n^{3/2}$, for some constant~$c$, but these asymptotic behaviours are not
incompatible with algebraicity~\cite{flajolet-context-free}. We present a more
sophisticated argument, reducing the transcendence question to linear algebra.
Our reasoning is inspired by Section~9 in~\cite{Hoeij97}. The main idea can be
traced back at least to~\cite{Ohtsuki82}; similar arguments are used
in~\cite{Chudnovsky80,BeBe85} and \cite[\S2]{CoSiTrUl02}. The argument may be
viewed as a general technique for proving transcendence of D-finite power
series;\footnote{There exists an alternative algorithmic procedure based
on~\cite{Singer80}, that allows in principle to answer this
question~\cite{Singer14}. It involves, among other things, factoring linear
differential operators, and deciding whether a linear differential operator
admits a basis of algebraic solutions. However, this procedure would have a
very high computational cost when applied to our situation.} even though based
on ideas in~\cite{Hoeij97}, we write down the proof in detail, since we have
not spotted out such a proof elsewhere in the literature. The needed
background on linear differential equations can be found
in~\cite[\S4.4.1]{PuSi03} or \cite[\S20]{Poole1960}.

\begin{prop}\label{prop:trans}
   The complete generating function $Q(x,y;t)$ for the model $\cS_1$
  is not algebraic.
\end{prop}
\begin{proof}
It is sufficient to prove that the series
\[Q(1,0;t) = 1+t+4t^2+8t^3+39t^4+98t^5 + O(t^{6})\] 
is not algebraic.

Assume by contradiction that $Q(1,0;t)$ is algebraic. Recall that it is
D-finite, and denote by $\mathcal{M}(t,D_t)$ the monic linear differential
operator in $\mathbb{Q}(t)[D_t]$ of minimal order such that
$\mathcal{M}(Q(1,0;t))=0$. Denote by $\mathcal{L}$ the 11th order linear
differential operator in $\mathbb{Q}(t)[D_t]$ computed in the proof of
Proposition~\ref{prop:S1-D-finite} such that $\mathcal{L}(Q(1,0;t))=0$. By
minimality, $\mathcal{M}$ right-divides $\mathcal{L}$ in $\mathbb{Q}(t)[D_t]$.
Moreover, $\mathcal{M}$ strictly right-divides $\mathcal{L}$, since otherwise
$\mathcal{L}$ would have a basis of algebraic
solutions~\cite[\S2]{CoSiTrUl02}, which is ruled out by a local analysis
showing that logarithms occur in a basis of solutions of $\mathcal{L}$ at
$t=0$. Therefore, the order $n$ of $\mathcal{M}$ is at most 10. On the other
hand, it follows from~\cite[\S2.2]{CoSiTrUl02} that $\mathcal{M}$ is Fuchsian
and has the form
\begin{equation} \label{eq:M10}
	\mathcal{M} = D_t^n + \frac{a_{n-1}(t)}{A(t)}D_t^{n-1} + \cdots + \frac{a_0(t)}{A(t)^n},
\end{equation}	
where $A(t)$ is a squarefree polynomial and $a_{n-i}(t)$ is a polynomial of
degree at most $\deg(A^i)-i$.

To get a contradiction (thus proving that $Q(1,0;t)$ is transcendental), we
use the following strategy: (i) get a bound on the degree of $A$ ---and thus
on the degrees of the coefficients of $\mathcal{M}$--- in terms of $n$ and of
local information of $\mathcal{L}$; (ii) use this bound and linear algebra to
show that such an operator $\mathcal{M}$ cannot annihilate $Q(1,0;t)$.

For (i), it is sufficient to bound the number of singularities of
$\mathcal{M}$ (i.e., the roots of $A(t)$, and possibly $t=\infty$). Let us
write $A(t) = A_1(t)A_2(t)$, where the roots of $A_1(t)$ are the finite
\emph{true} singular points of $\mathcal{M}$, and the roots of $A_2(t)$ are
the finite \emph{apparent} singular points of $\mathcal{M}$ (these are the
roots $p \in \mathbb{C}$ of $A(t)$ such that the equation
$\mathcal{M}(y(t))$=0 possesses a basis of analytic solutions at $t=p$).

Since $\mathcal{M}$ right-divides $\mathcal{L}$, the true singularities of
$\mathcal{M}$ form a subset of the true singularities of $\mathcal{L}$. Now,
$\mathcal{L}$ is explicitly known, and its local analysis shows that it is
Fuchsian and that it admits the following local data, presented below as pairs
(point $p$, (sorted) list of exponents of $\mathcal{L}$ at~$p$):
\begin{align*}
0, & \quad [-2, -2, -3/2, -1, -1/2, 0, 1, 2, 3, 4, 5], \\
1, & \quad [-1, 0, 1, 2, 3, 4, 5, 6, 7, 8, 9], \\
\infty, & \quad [1, 5/3, 2, 7/3, 8/3, 3, 10/3, 11/3, 4, 13/3, 5], \\
\pm \frac{\sqrt{3}}{9}, & \quad [0, 1, 3/2, 2, 5/2, 3, 3, 4, 5, 6, 7], \\
3 \pm 2 \sqrt{2}, & \quad [0, 1/2, 1, 3/2, 2, 5/2, 3, 4, 5, 6, 7], \\
55 \text{\, apparent singularities}, & \quad [0, 1, 2, 3, 4, 5, 6, 7, 8, 9, 11]. 
\end{align*}
Therefore, $\mathcal{L}$ has $6$ finite true singularities, and this implies that $\deg A_1$ is at most 6. It remains to upper bound $\deg A_2$; this will result from an analysis of the apparent singularities of $\mathcal{M}$. To do this, 
for any point $p$ in $\mathbb{C} \cup \{\infty\}$ we denote by
$S_p(\mathcal{M})$ the following quantity:  
	\[S_p(\mathcal{M}) = (\text{sum of local exponents of
          \,} \mathcal{M} \text{\, at \,} p) - (0 + 1 + \cdots +
        (n - 1)).\] 
If $p$ is an ordinary point of $\mathcal{M}$, 
then $S_p(\mathcal{M}) = 0$; therefore
$S_p(\mathcal{M})$ is an interesting quantity only for
singularities $p$ of $\mathcal{M}$. 
Moreover, the quantities $S_p(\mathcal{M})$ are globally related
by the so-called Fuchs' relation~\cite[\S4.4.1]{PuSi03}:
\begin{equation}\label{eq:Fuchs}
\sum_{p \in \mathbb{C} \cup \{\infty\}} S_p(\mathcal{M}) = \sum_{p \text{\,
singularity of \,} \mathcal{M}} S_p(\mathcal{M}) = \; - n(n-1).
\end{equation}	

If $p$ is an apparent singularity of $\mathcal{M}$, then $S_p(\mathcal{M}) \geq 1$, since the local exponents of $\mathcal{M}$ at $p$ are $n$ distinct non-negative integers, the largest one being at least equal to~$n$~\cite[Lemma~2.2]{CoSiTrUl02}.
We use this observation in conjunction with Fuchs'
relation~\eqref{eq:Fuchs} in order to bound the number of
apparent singularities of $\mathcal{M}$:
\begin{equation}\label{eq:app-bound}
\#\{\text{ apparent singularities of } \mathcal{M}  \, \}  \; \leq \; -n(n-1) - \sum_{p \text{\, true singularity of \,} \mathcal{M}} S_p(\mathcal{M}).
\end{equation}
For any true singularity $p$ of $\mathcal{M}$, the  list of local exponents of $\mathcal{M}$ at $p$ is a sublist of the list of exponents of $\mathcal{L}$ at $p$. Therefore  $S_p(\mathcal{M}) \geq S_p^{(n)}(\mathcal{L})$, where $S_p^{(n)}(\mathcal{L})$ denotes the sum of the smallest~$n$ elements from the list of exponents of $\mathcal{L}$ at~$p$, minus $(0 + 1 + \cdots + (n - 1))$. Thus, the sum $\displaystyle{\sum_{p \text{\, true singularity of \,} \mathcal{M}} S_p(\mathcal{M})}$ is minimal if the true singularities of $\mathcal{M}$ are precisely those true singularities $p$ of $\mathcal{L}$ for which $S_p^{(n)}(\mathcal{L})$ is non-positive.
It then follows from~\eqref{eq:app-bound} that 
\begin{equation}\label{eq:A2bound}
\deg(A_2)  \; \leq \; -n(n-1) - \sum_{p \text{\, true singularity of \,} \mathcal{L}} \min(0, S_p^{(n)}(\mathcal{L})).
\end{equation}

For instance, for $n=10$, the quantities $S_p^{(n)}(\mathcal{L})$ where $p$
runs through the set $\{ 0, 1, \infty, \frac{\sqrt{3}}{9},
-\frac{\sqrt{3}}{9}, 3 + 2 \sqrt{2}, 3 - 2 \sqrt{2}\}$ of true singularities
of $\mathcal{L}$ are equal respectively to $-42, -10, -17, -17, -17, -39/2,
-39/2$. Therefore, the inequality~\eqref{eq:A2bound} reads $\deg(A_2) \leq -
10 \times 9 - (42 + 10 + 17 + 17 + 17 + 39/2 + 39/2) = 52$, and thus $\deg(A)
= \deg(A_1) + \deg(A_2)\leq 6+52 = 58.$

A similar computation for $n=1, 2, \ldots, 9$ implies the respective upper
bounds $9, 12, 15, 19, 24, 31, 40, 47, 53$ for the degree of $A$. We conclude
that the minimal-order annihilating operator~$\mathcal{M}$ has the
form~\eqref{eq:M10} with $n\leq 10$ and coefficients of degree $\deg(A)\leq
58$ and $\deg(a_{n-i}) \leq 57i$ for $0\leq i \leq n-1$.

To summarise, under the assumption that $Q(1,0;t)$ is algebraic, we proved
that there exists a non-zero linear differential operator that annihilates it
(namely $A(t)^n \cdot \mathcal{M}$), of order at most $10$ and polynomial
coefficients of degrees at most~$580$. Moreover, this operator writes
$\hat{\mathcal{M}} = \sum_{i=0}^{10} \sum_{j=0}^{570+i} m_{i,j}t^j D_t^i$. The
equality $\hat{\mathcal{M}}(Q(1,0;t)) = 0 \bmod t^{6400}$ then yields a linear
system of 6400 equations in the 6336 unknowns $m_{i,j}$, and a linear algebra
computation shows that its unique solution is the trivial one $m_{i,j}=0$ for
all $i$ and $j$, contradicting the fact that $\mathcal{M} \neq 0$.

In conclusion, $Q(1,0;t)$ is transcendental.
\end{proof}

\section{Final comments and questions}
\label{sec:final}

We have determined strategies for classifying the generating functions of
walks with small steps in the positive octant, based on the algebraic kernel
method, reductions to lower dimensions, and computer algebra.

In this way, we have been able to prove D-finiteness for all models for which
we had guessed a differential equation. All these models have a finite group.
However, we have not ruled out the possibility of non-D-finite models having a
finite group. Indeed, for the 19 models discussed in
Section~\ref{sec:nonHzero} and listed in Figure~\ref{fig:19}, the group is
finite but we have not discovered any differential equations. If one of these
models was proved to be non-D-finite, it would be the end of the
correspondence between finite groups and D-finiteness that holds for quadrant
walks. One approach to prove non-D-finiteness might be a 3D extension
of~\cite{BoRaSa12}.

Here are other natural questions that we leave open.
\begin{itemize}
\item {\bf{Groups.}} Can one prove that the groups that we conjecture infinite
are infinite indeed? Some techniques that apply to quadrant models are
presented in~\cite[Section~3]{BoMi10}.
\item {\bf{Transcendence.}} For the complete generating function of many
models with a finite group, we have proved D-finiteness but not algebraicity.
Can one prove that these models are in fact transcendental, either using an
asymptotic argument or by adapting the proof of Proposition~\ref{prop:trans}?
\item {\bf Non-D-finiteness.} Can one prove the non-D-finiteness of certain
quadrant models with repeated steps with the technique of~\cite{BoRaSa12}? Can
it be extended to octant models?
\item {\bf Human proofs.} Can one give ``human proofs'' for the computer
algebra results of Section~\ref{sec:computer}?
\item {\bf Repeated steps.} If we investigate systematically quadrant models
with repeated steps (and arbitrary cardinality), do we find more attractive
models, like those of Figure~\ref{fig:quatre}?
\item {\bf{Closed form expressions.}} Can one express all the D-finite length
generating functions that we obtained in terms of hypergeometric series, as
was done for quadrant walks in~\cite{BCHKP}?
\end{itemize}


\bigskip \noindent {\bf Acknowledgements.} S.M. would like to thank Simon
Fraser University and the Inria-MSR joint lab, where he was working while the
majority of this research was completed.


\bibliographystyle{plain}

\spacebreak
\appendix

\section{The number of non-equivalent 2D or 3D models}
\label{app:interesting}

We now want to prove Proposition~\ref{prop:interesting}. To begin
with, we determine the polynomial counting models with no unused step.
\subsection{Models with no unused step}
\label{sec:IsoRed-A}

\begin{prop}\label{propA}

The generating function of step sets that contain no unused step, counted
up to permutations of the coordinates, is
  \begin{multline*}
J=
1+3u+21u^2+179u^3+1294u^4+7041u^5+28917u^6+92216u^7+235338u^8\\
+492509u^9+860520u^{10}+1271528u^{11}+1603192u^{12}+1734397u^{13}\\
+1614372u^{14}+1293402u^{15}+890395u^{16}+524638u^{17}+263008u^{18}\\
+111251u^{19}+39256u^{20}+11390u^{21}+2676u^{22}+500u^{23}+73u^{24}+9u^{25}+u^{26}.
 \end{multline*}
\end{prop}
We prove this using Burnside's lemma~\cite[Lem.~7.24.5]{stanley-vol2}:
\begin{equation}\label{J-burnside}
 J=\frac 1 6 \sum_{\sigma \in \Sn_3} J^{\sigma},
\end{equation}
where $J^{\sigma}$ counts sets that contain no unused step and are left
invariant under the \p\ $\sigma$. By symmetry, it suffices to determine
$J^{\sigma}$ when $\sigma=\id$, when $\sigma$ is the 2-cycle $(1,2)$ and when
$\sigma$ is the 3-cycle $(1,2,3)$. This is done in
Sections~\ref{sec:IsoRed-Pid},~\ref{sec:IsoRed-Pxy},
and~\ref{sec:IsoRed-Pxyz}, respectively.

\subsubsection{No prescribed symmetry: the series $\boldsymbol{J^{\id}}$}
\label{sec:IsoRed-Pid}

There are $26=3^3-1$ non-trivial steps, and thus the generating function of
all models is $(1+u)^{26}.$

Now we want to remove sets that contain unused steps, using the
characterisation of Lemma~\ref{lem:unused}. We use the notation $A_x$ for the
generating function of models that satisfy Condition $(A_x)$, and so on. We
extrapolate this so that $A_{xy}$ counts models satisfying $(A_x)$ and the
condition $(A_y)$ obtained from $(A_x)$ by replacing $x$ by $y$. Similarly,
$A_xC_z$ counts models satisfying both $(A_x)$ and $(C_z)$, and is \emm not,
the product of the series $A_x$ and $C_z$. We hope this will not raise any
confusion. Let us note that
$$
A_zC_z=BC_z= C_yC_z=  0 \quad \hbox{and}\quad A_{xyz}=A_{xyz}B.
$$
Using inclusion-exclusion,  obvious symmetries ($A_x=A_y$ for
instance) and the above  relations, we obtain
\begin{equation}\label{JIE}
J^{\id}=
(1+u)^{26}-3A_x-B-3C_z+3A_{xy}+3A_xB+6A_xC_z-3A_{xy}B-3A_{xy}C_z
.
\end{equation}

Let us now determine the series $A_x$, $B$ and $C_z$. We will then give
without details the values of the other five series.

For $A_x$, we count models containing no $x^+$-step, but at least one
$x^-$-step. Among the 26 steps, $26-9=17$ are not $x^+$. Among these 17
steps, 8 are not $x^-$ either. Hence
$$
A_x=(1+u)^{17}-(1+u)^8 =[17]-[8],
$$
where we denote $[i]=(1+u)^i$.

For $B$, we count non-empty models in which each step has a negative
coordinate. There are $26-7=19$ such steps, and thus
$$
B=[19]-1.
$$

Finally, a model satisfying $(C_z)$ must contain $001$ and consists
otherwise of steps $ijk$ with $i+j \le 0$. There are $26-9=17$ steps
satisfying this inequality, among which $001$, which is necessarily in
$\cS$. Thus
$$
C_z=u\left( [16]-[1]\right)
$$
as the model must not be a subset of $\{ 001, 00\bar 1\}$.

Let us now list the values of the other series occurring in~\eqref{JIE}:
\begin{eqnarray*}
  A_{xy}&=&[11]-2[5]+[2]
\\
A_xB&=& [14]-[5]
\\
A_xC_z&=&u\left( [13]-[4]\right)
\\
A_{xy}B&=&[10]-2[4]+[1]
\\
A_{xy}C_z&=&u\left( [10]-2[4]+[1]\right).
\end{eqnarray*}
Returning to~\eqref{JIE}, this now gives the generating function of models with no
unused step:
\begin{multline}
  \label{Jid}
J^{\id}=[26]-[19]-3[17]+3[14]+3[11]-3[10]+3[8]-9[5]+6[4]+3[2]-3[1]+1
\\+3u\left( -[16]+2[13]-[10]\right).
\end{multline}

\subsubsection{With an $xy$-symmetry: the series $\boldsymbol{J^{(1,2)}}$}
\label{sec:IsoRed-Pxy}
We now count models $\cS$ with no unused step that have  an
$xy$-symmetry:  $ijk\in \cS \Leftrightarrow jik\in \cS$. We
revisit the above argument with this additional constraint. 
The set of all 26 steps has 17 orbits under the action of the
  2-cycle $(1,2)$: 8 of them are singletons, and 9 are pairs. Hence
  the generating function of $xy$-symmetric sets is $ [8][[9]]$
where $[[j]]$ stands for $ (1+u^2)^j$.
We recycle the notation $A_x$, $B$, and so on, but we add bars to
indicate that we only count models with an $xy$-symmetry:
$\overline{A_x}, \overline B$, etc. We apply again inclusion-exclusion to count symmetric models with no unused step,
noting that
$$
\overline{A_x}=\overline{A_y}=\overline{A_{xy}},  \quad 
\overline{C_x}=\overline{A_zC_z}=\overline{BC_z}=0 
\quad \hbox{and} \quad \overline{A_{xz}B}= \overline{A_{xz}}.
$$
We thus obtain:
$$
J^{(12)}= [8][[9]]-\overline{A_x}-\overline{A_z}-\overline B- \overline{C_z}
+\overline{A_xB} +\overline{A_zB} +
\overline{A_xC_z}.
$$
Let us determine $\overline{A_x}$.  We observe that 5 singleton orbits are
not $x^+$ (among which 2 are not $x^-$ either), and that 3 orbits of
size 2 are not $x^+$ (all of them are $x^-$). This gives
$$
\overline{A_x}= [5][[3]]-[2].
$$
For  $\overline{A_z}$, we find that  5 singleton orbits are
not $z^+$ (among which 2 are not $z^-$ either), and that 6 orbits of
size 2 are not $z^+$ (among which 3 are not $z^-$ either). This gives
$$
\overline{A_z}= [5][[6]]-[2][[3]].
$$
We further obtain
\begin{eqnarray*}
   \overline B &=&[5][[7]]-1
\\
 \overline{C_z}&=&u([4][[6]]-[1])
\\
 \overline{A_xB} &=&[4][[3]]-[1]
\\
 \overline{A_zB} &=&[4][[5]]-[1][[2]]
\\
\overline{A_xC_z}&=&u([4][[3]]-[1]),
\end{eqnarray*}
so that finally
\begin{multline}\label{J12}
  J^{(1,2)}= [8][[9]]-[5]\left( [[7]]+[[6]]+[[3]]\right)
+[4]\left( [[5]]+[[3]]\right) +[2]\left( [[3]]+1\right)
\\-[1]\left([[2]]+1\right)+1 -u[4]\left( [[6]]-[[3]]\right).
\end{multline}

\subsubsection{With an $(x,y,z)$-symmetry: the series $\boldsymbol{J^{(1,2,3)}}$}
\label{sec:IsoRed-Pxyz}
We finally count how many models with no unused step
are left invariant under the action of the cycle
$(1,2,3)$. That is, if $ijk\in \cS$, then $jki$ (and
$kij$) also belong to $\cS$. The set of our 26 steps has 10 orbits
under the action of the 3-cycle: 2 of them are singletons, and 8
contain 3 elements. Hence the generating function of symmetric models is $[2]\langle8\rangle$,
where $\langle j\rangle=(1+u^3)^j$.  We now use tildes above our generating functions to indicate the
invariance under $(1,2,3)$. Noting that
$$
\widetilde{A_x}=\widetilde{A_y }=\widetilde {A_z}=\widetilde {A_xB}
 \quad \hbox{and} \quad \widetilde{C_z}=0,
$$
the inclusion-exclusion reduces to
\begin{equation}\label{J123}
J^{(1,2,3)}=[2]\langle  8\rangle-\widetilde B
= [2]\langle8\rangle-[1]\langle6\rangle+1,
\end{equation}
since there are 1 non-positive singleton orbit  and 6
non-positive orbits of size 3.

\noindent
\begin{proof}[Proof of Proposition~\ref{propA}] We apply Burnside's
  formula~ \eqref{J-burnside} with $J^{\id}$, $J^{(1,2)}$ and
  $J^{(1,2,3)}$ given by~\eqref{Jid}, \eqref{J12} and~\eqref{J123} respectively.
  
\end{proof}
\subsection{Models with no unused step and dimension at most 1}
\label{sec:IsoRed-A1}
We now establish the counterpart of Proposition~\ref{propA} for models of
dimension at most 1.

\begin{prop}
\label{propA1}
  The generating function of models of dimension at most $1$, having
  no unused step, counted up to permutations of the coordinates, is
  \begin{multline*}
    K= 1+3\,u+21\,{u}^{2}+106\,{u}^{3}+315\,{u}^{4}+616\,{u}^{5}+846\,{u}^{6}
+844\,{u}^{7}\\+622\,{u}^{8}+341\,{u}^{9}+138\,{u}^{10}+40\,{u}^{11}+8\,
{u}^{12}+{u}^{13}.
  \end{multline*}
\end{prop}
As before, we  prove this using Burnside's lemma:
\begin{equation}\label{K1-burnside}
 K=\frac 1 6 \sum_{\sigma \in \Sn_3} K^{\sigma},
\end{equation}
where $K^{\sigma}$ is the generating function of 0- or 1-dimensional models with no
unused step, left invariant by the permutation $\si$. Again, we only
need to determine $K ^{\id}$, $K^{(1,2)}$ and $K^{(1,2,3)}$.

\subsubsection{Small dimension, no prescribed symmetry: the series $\boldsymbol{K^{\id}}$}
\label{sec:1Dno}
By definition,  a model is at most one-dimensional if it suffices to
enforce one non-negativity condition to confine walks in the octant. We
denote by $K_x^{\id}$ the generating function of models (with no unused step) for which it
suffices to enforce the $x$-condition. We denote  similarly by $K_{xy}^{\id}$
the generating function of models for which it suffices to enforce the $x$-condition or
the $y$-condition, and adopt a similar notation $K_{xyz}^{\id}$ for models
for which it suffices to enforce any of the three
conditions. This is the case, for instance, if $\cS=\{\bone \bone
\bone, 111\}$. By inclusion-exclusion,
\begin{equation}\label{K1id-xyz}
K^{\id}= 3K_x^{\id}-3K_{xy}^{\id}+K_{xyz}^{\id}.
\end{equation}

\medskip
\paragraph{\bf The generating function $\boldsymbol {K_x^{\id}}$.} This polynomial count models
with no unused step satisfying Lemma~\ref{lem:1D}. To begin with, let us
count \emm all, models satisfying this lemma, or equivalently, one
(and exactly one) of the following conditions:
\begin{enumerate}
\item[1.]  there is no $y^-$ nor $z^-$ step,
\item[2.] there is no $y^-$ step, but there are $z^-$ steps and every
  step $ijk$ satisfies $k\ge i$,
\item[3.]  there 
is no $z^-$ step, but there are $y^-$ steps and every
  step $ijk$ satisfies $j\ge i$,
\item[4.]  there  are $y^-$ and $z^-$ steps and every
  step $ijk$ satisfies $j \ge i$ and $k\ge i$.
\end{enumerate}
Since exactly 11 of the 26 steps are neither $y^-$ nor $z^-$, the
generating function of models satisfying Condition~1 is $(1+u)^{11}=[11]$.  For
models satisfying Condition~2, there are also 11 admissible steps,
among which 9 are not $z^-$. Hence the associated generating function is
$[11]-[9]$. Of course, the argument and the series are the same for
Condition~3. Finally, there are 13 admissible steps for Condition~4,
among which 10 are not $z^-$, 10 are not $y^-$ and 8 are neither $y^-$
nor $z^-$. Hence the generating function of models satisfying Condition~4 is
$[13]-2[10]+[8]$. In total, this gives
$$
 [13]+3[11]-2[10]-2[9]+[8]
$$
for the generating function of 
models satisfying Lemma~\ref{lem:1D}. We now want to subtract those that
have unused steps. We use  the notation $A_x$ for the generating function of
models that satisfy Lemma~\ref{lem:1D} and Condition $(A_x)$ of
Lemma~\ref{lem:unused}. Similarly, $C_z$ counts models that satisfy Lemma~\ref{lem:1D}
  and Condition $(C_z)$ of Lemma~\ref{lem:unused}, and so on. Observe that
$$
A_y=A_{xy}, \quad BA_x=B,  \quad  BA_{yz}= A_{yz}, \quad C_x=0   \quad
\hbox{and} \quad A_xC_z=C_z.
$$
Moreover, $y$ and $z$ play symmetric roles. Hence, the
inclusion-exclusion formula reduces in this case to 
$$
K_x^{\id}=[13]+3[11]-2[10]-2[9]+[8] -A_x.
$$
Let us determine $A_x$, that is, the generating function of models that contain $x^-$
steps but no $x^+$ steps, and satisfy one of the 4 conditions above.
For those that satisfy Condition 1, we find the polynomial $[7]-[3]$
(because there are 7 steps that are neither $x^+$, nor $y^-$ nor
$z^-$, among which 3 are not $x^-$ either). For Condition~2, we obtain
$[9]-[7]$ and for Condition~3 as well, by symmetry. Finally, for
Condition~4 we obtain $[12]-2[9]+[7]$. Finally,
$$
A_x=[12]-[3].
$$
Hence
\begin{equation}\label{Kxid}
K_x^{\id}=[13]-[12]+3[11]-2[10]-2[9]+[8]+[3].
 \end{equation}

\medskip
\paragraph{\bf The generating function $\boldsymbol {K_{xy}^{\id}}$.} This polynomial
counts models with no unused step   that satisfy Lemma~\ref{lem:1D}
and the counterpart of Lemma~\ref{lem:1D} with $x$ replaced by
$y$. Rather than proceeding by inclusion-exclusion as above, we find
convenient to list all these models $\cS$:
\begin{itemize}
\item Either $\cS$ only consists of steps from $\{0,1\}^3\setminus \{000\}$; the generating function
  is then $[7]$.
\item If there is a $z^-$ step in $\cS$, then Lemma~\ref{lem:1D} and
  its $y$-counterpart imply that it must be $\bone\bone\bone$. Since this step is
  not unused, there must be an $x^+$ step $1jk$, but then
  Lemma~\ref{lem:1D} implies that it is $111$. There is no other $x^+$
  nor $y^+$ step. All remaining steps $ijk$ must satisfy $i=j\le k$,
  and so must be taken in $\{001, \bone\bone0, \bone\bone 1\}$. Hence
  the associated generating function is $u^2[3]$.
\item Otherwise there is no $z^-$ step, and there is, say, a $y^-$
  step. By Lemma~\ref{lem:1D}, it must be of the form $\bone\bone k$
  with $k=0,1$. Since this step is not unused, there must be also an
  $x^+$ step, which must be of the form $11\ell$, with $\ell=0,1$. The
  only other possible step is $001$. Hence the generating function for this case is
  $\left([2]-1\right)^2[1]$. 
\end{itemize}
In total,
$$
K^{\id}_{xy}=[7]+u^2[3]+\left([2]-1\right)^2[1].
$$

\medskip
\paragraph{\bf The generating function $\boldsymbol {K_{xyz}^{\id}}$.} This polynomial
counts models with no unused step   that satisfy Lemma~\ref{lem:1D}
and the counterparts of Lemma~\ref{lem:1D} with $x$ replaced by
$y$ and then by $z$. It suffices to select from the list established
above for $K_{xy}^{\id}$ the models satisfying also the latter
condition. One finds that only models with non-negative steps, and
also $\cS=\{\bone\bone\bone,111\}$ satisfy the required conditions, so
that
$$
K^{\id}_{xyz}=[7]+u^2.
$$
Returning to~\eqref{K1id-xyz}, we obtain the generating function of models of dimension at
most 1 having no unused steps:
\begin{equation}\label{K1idsol}
K^{\id}=3[13]-3[12]+9[11]-6[10]-6[9]+3[8]-2[7]+3[3]-3u^2[3]-3\left([2]-1\right)^2[1] +u^2.
\end{equation}

\subsubsection{Small dimension, $xy$-symmetry: the series $\boldsymbol{K^{(1,2)}}$}
We revisit the enumeration of Section~\ref{sec:1Dno} and enforce now an
$xy$-symmetry. The variables $x$ and $y$ still play the same role, but
$z$ plays a different role, and~\eqref{K1id-xyz} becomes
$$
K^{(1,2)}= 2 K_x^{(1,2)}+K_z^{(1,2)}-K_{xy}^{(1,2)}-2K_{xz}^{(1,2)}+
K_{xyz}^{(1,2)}.
$$
We note however that the $xy$-symmetry implies that
$$
K_x^{(1,2)}= K_{xy}^{(1,2)}\quad \hbox{and}\quad
K_{xz}^{(1,2)}=
K_{xyz}^{(1,2)},
$$
so that
\begin{equation}\label{K112-xyz}
K^{(1,2)}= K_{xy}^{(1,2)}+K_z^{(1,2)}-K_{xyz}^{(1,2)}.
\end{equation}

\medskip
\paragraph{\bf The generating function $\boldsymbol {K_z^{(1,2)}}$.}
Clearly, $K_z^{(1,2)}=K_x^{(2,3)}$. We
revisit the determination of $K_x^{\id}$ made in  Section~\ref{sec:1Dno}, but
restricting the enumeration to models with a $yz$-symmetry. We begin
with the models that satisfy Lemma~\ref{lem:1D} and have a
$yz$-symmetry. With Condition~1 we find the generating function
$[5][[3]]$. Conditions~2 and~3 do not satisfy $yz$-symmetry, so do not
contribute. Condition~4 contributes $[5][[4]]-[4][[2]]$. 

As in the determination of $K_x^{\id}$, it remains to subtract the
series $\overline{A_x}$ corresponding to models that contain $x^-$
steps, but no $x^+$ step. With Condition~1 we find
$[3][[2]]-[1][[1]]$, and with Condition~4, $[4][[4]]-[3][[2]]$.
Hence
$$
K_z^{(1,2)}= [5][[3]]+[5][[4]]-[4][[2]]
+[1][[1]]-[4][[4]].
$$

\medskip
\paragraph{\bf The generating function $\boldsymbol {K_{xy}^{(1,2)}}$.}
We now revisit the list of models used to determine $K_{xy}^{\id}$ in
Section~\ref{sec:1Dno}, now enforcing an  $xy$-symmetry. The generating function of
symmetric non-negative models is $[3][[2]]$. All listed models that
include a negative step are $xy$-symmetric. Hence
$$
K_{xy}^{(1,2)}= [3][[2]]+u^2[3]+ ([2]-1)^2 [1].
$$

\medskip
\paragraph{\bf The generating function $\boldsymbol {K_{xyz}^{(1,2)}}$.}
We now revisit the determination of $K_{xyz}^{\id}$ by  enforcing an
$xy$-symmetry. Besides non-negative models, counted by $[3][[2]]$, we
only have the model $\{\bone\bone\bone, 111\}$. Hence
$$
K_{xyz}^{(1,2)}= [3][[2]]+u^2.
$$

Returning to~\eqref{K112-xyz}, we obtain
\begin{equation}\label{K12sol}
K^{(1,2)}= [5][[3]]+[5][[4]]-[4][[2]]+[1][[1]]-[4][[4]] +u^2[3]+ ([2]-1)^2 [1]-u^2.
\end{equation}

\subsubsection{Small dimension, cyclic symmetry: the series $\boldsymbol{K^{(1,2,3)}}$}
We have finally reached the last, and simplest step of our
calculation. By symmetry, we have
$$
K^{(1,2,3)}=K_{xyz}^{(1,2,3)}.
$$
Hence we now revisit the determination of $K_{xyz}^{\id}$ by  enforcing a cyclic
symmetry. Besides non-negative models, counted by $[1]\langle 2\rangle$, we
only have the model $\{\bone\bone\bone, 111\}$. Hence
\begin{equation}\label{K123sol}
K^{(1,2,3)}=[1]\langle 2\rangle+u^2.
\end{equation}

We can now conclude the
\begin{proof}[Proof of Proposition~\ref{propA1}]
  Apply Burnside's formula~\eqref{K1-burnside} with $K^{\id}$, $K^{(1,2)}$ and
  $K^{(1,2,3)}$ respectively given  by~\eqref{K1idsol}, \eqref{K12sol} and ~\eqref{K123sol}.
\end{proof}

We have finally counted how many  models we need to study.

\begin{proof}[Proof of Proposition~\ref{prop:interesting}]
The series $I$ is $J-K$, with $J$ and $K$ given by Propositions~\ref{propA}
and~\ref{propA1} respectively.  
\end{proof}
\end{document}